\documentclass[11pt]{amsart}

\usepackage{mathpazo}
\usepackage{amsmath,amssymb}
\usepackage{graphicx}
\usepackage{import}
\usepackage{amscd}
\usepackage{wrapfig}
\usepackage{fullpage}
\usepackage{epsfig}
\numberwithin{equation}{section}

\newtheorem{theorem}{Theorem}[section]

\newtheorem{lemma}[theorem]{Lemma}
\newtheorem{corollary}[theorem]{Corollary}
\newtheorem{proposition}[theorem]{Proposition}

\newtheorem{remark}[theorem]{Remark}

\newcommand{\D}{{\mathbb D}}
\newcommand{\C}{{\mathbb C}}

\newcommand{\N}{{\mathbb N}}

\newcommand{\R}{{\mathbb R}}
\newcommand{\Z}{{\mathbb Z}}
\newcommand{\T}{{\mathcal T}}
\renewcommand{\H}{{\mathbb H}}

\newcommand{\cC}{\mathcal{C}}

\newcommand{\teichmuller}{Teichm{\"u}ller{ }}

\makeatletter
 \let\c@theorem=\c@subsection
 \let\c@conjecture=\c@subsection
 \let\c@lemma=\c@subsection
 \let\c@proposition=\c@subsection
 \let\c@claim=\c@subsection
 \let\c@question=\c@subsection
 \let\c@criterion=\c@subsection
 \let\c@vfconj=\c@subsection
 \let\c@definition=\c@subsection
 \let\c@notation=\c@subsection
 \let\c@remark=\c@subsection
 \let\c@example=\c@subsection
 \let\c@equation=\c@subsection
 \let\c@figure=\c@subsection
 \let\c@wrapfigure=\c@subsection

\makeatother

\begin{document}

\title{Partial sums of excursions along random geodesics and volume asymptotics for thin parts of moduli spaces of quadratic differentials.}
\author{Vaibhav Gadre}
\address{Mathematics Institute, Zeeman Building, University of Warwick, Coventry, CV47AL.}
\email{gadre.vaibhav@gmail.com}
 
 
\maketitle

\begin{abstract}
For a non-uniform lattice in $SL(2, \R)$, we consider excursions in cusp neighborhoods of a random geodesic on the corresponding finite area hyperbolic surface or orbifold. We prove a strong law for a certain partial sum involving these excursions. This generalizes a theorem of Diamond and Vaaler for continued fractions \cite{Dia-Vaa}. In the \teichmuller setting, we consider invariant measures for the $SL(2, \R)$ action on the moduli spaces of quadratic differentials. By the work of Eskin and Mirzakhani \cite{Esk-Mir}, these measures are supported on affine invariant submanifolds of a stratum of quadratic differentials. For a \teichmuller geodesic random with respect to a $SL(2,R)$-invariant measure, we study its excursions in thin parts of the associated affine invariant submanifold.  Under a regularity hypothesis for the invariant measure, we prove similar strong laws for certain partial sums involving these excursions. The limits in these laws are related to the volume asymptotic of the thin parts. By Siegel-Veech theory, these are given by various Siegel-Veech constants. As a direct consequence, we show that the word metric grows faster than $T \log T$ along  \teichmuller geodesics random with respect to the Masur-Veech measure.
\end{abstract}

\pagestyle{empty}

\section{Introduction}
The aim of this paper is to provide a specific analogy between non-uniform lattices in $SL(2, \R)$ and mapping class groups. This analogy is established from the point of view of cusp excursions of random geodesics on the quotient hyperbolic surface on one hand and cusp excursions of random \teichmuller geodesics in a $SL(2,\R)$ orbit closure in a stratum of the moduli space of quadratic differentials on the other.

Let $G$ be a non-uniform lattice in $SL(2, \R)$, i.e. the quotient $X = G \backslash \H^2$ is a complete finite area surface/orbifold with finitely many cusps $c_1, \dots, c_J$. Let $X_{\text{cusps}}$ denote the union of disjoint horoball neighborhoods of the cusps. The lift of $X_{\text{cusps}}$ to $\H^2$ is a countable collection $\mathcal{H}$ of disjoint horoballs. The complement $X \setminus X_{\text{cusps}}$ is a compact set $X_{\text{thick}}$ called the thick part of $X$. The complement of the horoballs in $\mathcal{H}$ is the lift $\widetilde{X}_{\text{thick}}$ which we call the thick part in $\H^2$.

The unit tangent bundle $T^1 \H^2$ can be naturally identified with $SL(2, \R)$. It carries a natural $SL(2, \R)$-invariant measure which is simply the Haar measure. In the upper half-plane model the measure is given by 
\[
d \ell= \frac{dx\ dy\ d \theta}{2\pi y^2}.
\]
The geodesic flow is given by the action of the diagonal subgroup. So $d\ell$ descends to a flow invariant measure on $T^1 X$ and is called \emph{Liouville measure}. 
The conditional measure on the unit circle in the tangent plane at any point is the pullback via the visual map of the standard Lebesgue measure on
$\partial \H^2 = S^1$.

By ergodicity of the geodesic flow $d\ell$-almost every geodesic ray is recurrent to $X_{\text{thick}}$. Also by ergodicity and the fellow traveling property of hyperbolic geodesics, Lebesgue almost every geodesic ray from any base-point $x_0 \in X$ ventures into $X_{\text{cusps}}$ infinitely often. In particular, a geodesic ray $\gamma$ in $\mathbb{H}^2$ whose endpoint $r$ in $S^1$ is typical with respect to the Lebesgue measure enters and leaves infinitely many horoballs. By analyzing the collection $\mathcal{H}$, Sullivan \cite{Sul} showed that the $\limsup$ of maximum depth in the cusp neighborhoods that the geodesic ray achieves is asymptotically $(1/2) \log T$, where $T$ is the time along the geodesic. It is convenient to assume that $x_0 \in X_{\text{thick}}$ which can  be achieved by making the cusp neighborhoods smaller if necessary. 

To set up notation, let $\gamma(x_0, r)$ be the geodesic ray from $x_0$ to $r \in S^1$. We denote by $\gamma_T(x_0, r)$ the point on it distance $T$ from $x_0$. When the context is clear we will use just $\gamma$ and $\gamma_T$. Let $\pi: \H^2 \to \widetilde{X}_{\text{thick}}$ be the closest point projection. Let $N = N(T)$ be the number of horoballs that $\gamma$ intersects up to $\gamma_T$. We enumerate this collection of horoballs $\mathcal{H}(\gamma, T)= \{H_1, H_2, \dots, H_N\}$ in the order of increasing time. For all $k< N$, $\gamma$ enters and exits $H_k$; $H_N$ may be an exception if $\gamma_T \in H_N$. Let $d_{\text{thick}}$ be the path metric on $\widetilde{X}_{\text{thick}}$. For a horoball $H$ that $\gamma$ enters and exits, the {\em complete excursion} $E(\gamma, H)$ is defined as the $d_{\text{thick}}$-distance between the entry and exit points. If $\gamma_T \in H_N$ then the {\em partial excursion} $E(\gamma, H_N)$ is the $d_{\text{thick}}$-distance between the entry point for $H_N$ and $\pi(\gamma_T)$. 

The {\em total excursion} till time $T$ first defined in \cite{Gad-Mah-Tio} is given by
\[
E(\gamma, T) = \sum\limits_{k \leqslant N} E(\gamma, H_k). 
\]
It was shown using \cite[Proposition 5.4]{Gad-Mah-Tio} that along Leb-typical geodesic rays $E(\gamma, T)/T \to \infty$. We prove here:
\begin{theorem}\label{Esum}
For Lebesgue almost every $r$ in $S^1$, 
\[
\lim_{T \to \infty} \frac{E(\gamma,T) - \max\limits_{k \leqslant N} E(\gamma, H_k)}{T \log T} = \left(\frac{2}{\pi}\right)\frac{\ell(T^1 X_{\text{cusp}})}{\ell(T^1 X)}.
\]
\end{theorem}

\subsection{Continued fractions}
Let $r \in [0,1]$. The classical continued fraction expansion of $r$ is given by
\[
r = \cfrac{1}{a_1+ \cfrac{1}{a_2+ \cfrac{1}{+\dotsb }}} 
\]
where each $a_i$ is a positive integer. When $r$ is irrational the expansion is infinite. We denote the expansion as $[a_1, a_2, \dots]$.

\begin{theorem}[Diamond-Vaaler \cite{Dia-Vaa}]\label{Diam-Vaal}
For Leb-almost every $r \in [0,1]$
\[
\lim_{n \to \infty} \frac{\sum\limits_{i=1}^n a_i - \max\limits_{i\leqslant n} a_i}{n\log n} = \frac{1}{\log 2}.
\]
\end{theorem} 

We will derive Theorem \ref{Diam-Vaal} from Theorem \ref{Esum} as a special case when $X$ is the modular surface $SL(2, \Z) \backslash \H^2$. Excursions of geodesic rays into the cusp of $X$ are related to coefficients in the continued fraction expansion of the point at infinity for the geodesic. Diamond-Vaaler used techniques specific to the symbolic dynamics (Gauss map) in the theory of continued fractions. Theorem \ref{Esum} relies on more general features viz. asymptotic for $\text{vol}(X_{\text{cusps}})$ and exponential mixing of the geodesic flow. These features are also true for the \teichmuller geodesic flow in the setting of quadratic differentials. 

\subsection{Word metric along random geodesics}
We state some direct implications of Theorem \ref{Esum}. 

For distinct points $x, y \in \H^2$ let $\gamma(x, y)$ be the hyperbolic geodesic segment between them. The {\em projected path} $p(x,y)$ is defined by $p(x,y) = \pi(\gamma(x,y))$. Let $L(x, y)$ be the $d_{\text{thick}}$-length of $p(x,y)$. The quantify $L(x_0,\gamma_T) - E(\gamma, T)$ is the time spent by $\gamma$ in interior of $X_{\text{thick}}$. By ergodicity of geodesic flow, this grows linearly in $T$ and hence it follows:

\begin{theorem}\label{projpath}
For Leb-almost every $r \in S^1$
\[
\lim_{T \to \infty} \frac{L(x_0, \gamma_T) - \max\limits_{k \leqslant N} E(\gamma, H_k)}{T \log T} = \left(\frac{2}{\pi}\right)\frac{\ell(T^1 X_{\text{cusp}})}{\ell(T^1 X)}.
\]
\end{theorem}

The projected path $p(x_0, \gamma_T)$ is a quasi-geodesic in $(\widetilde{X}_{\text{thick}}, d_{\text{thick}})$ \cite[Lemma 5.1]{Gad-Mah-Tio}. More precisely, $L(x_0, \gamma_T) - d_{\text{thick}}(x_0, \gamma_T)$ grows at most linearly in $N$. As we show in Lemma \ref{n-T-lemma}, $N$ grows linearly in $T$. Hence, we get
\begin{theorem} \label{thick-metric}
For Leb-almost every $r \in S^1$
\[
\lim_{T \to \infty} \frac{d_{\text{thick}}(x_0, \gamma_T) - \max\limits_{k \leqslant N} E(\gamma, H_k)}{T \log T} = \left(\frac{2}{\pi}\right)\frac{\ell(T^1 X_{\text{cusp}})}{\ell(T^1 X)}.
\]
\end{theorem}

We say a basepoint $x_0 \in \mathbb{H}^2$ is generic if the stabilizer of $x_0$ in $G$ is trivial. The $G$-orbit of $x_0$ is called a lattice. If $x_0$ is a generic basepoint, then each lattice point corresponds to a unique group element. Assuming a generic basepoint, each point $\gamma_T$ has at least one closest lattice point $h_T x_0$. In fact, this closest point is unique for almost all points along
the geodesic.

The group $G$ is finitely generated. A finite choice of generators for $G$ defines a proper {\em word metric} $d_G$ on $G$. Different choices of generators produce quasi-isometric metrics. Let $d_G(1,h_T)$ be the word length for the lattice point $h_T x_0$ closest to $\gamma_T$. 

$G$ acts cocompactly on $\widetilde{X}_{\text{thick}}$. So by the Svarc-Milnor lemma, $(G,d_G)$ is quasi-isometric to $(\widetilde{X}_{\text{thick}}, d_{\text{thick}})$. Thus, a consequence of Theorem \ref{thick-metric} is the following theorem:

\begin{theorem}\label{wordmetric}
There exists a constant $M_1> 0$ that depends on the word metric such that for Leb-almost every $r\in S^1$
\[
M_1 T \log T  < d_G(1, h_T) 
\]
for $T$ sufficiently large depending on $r$. 
\end{theorem}

In fact, if the contribution from the largest excursion is removed, then the word metric grows like $T \log T$ up to uniform multiplicative and additive constants. Theorem \ref{wordmetric} should be thought of as a refinement of Proposition 5.6 in \cite{Gad-Mah-Tio} which states that along a Leb-generic geodesic ray the ratio $d_G(1,h_T)/T$ goes to infinity as $T \to \infty$.

\subsection{Moduli space of quadratic differentials} Let $S$ be a hyperbolic surface of finite type. $S$ is non-sporadic if it is not a sphere with at most four punctures or boundary components or a torus with at most one puncture or boundary component. In the sporadic examples, the \teichmuller space is either trivial or isometric to $\H^2$ and the mapping class group is a non-uniform lattice in $SL(2, \R)$. This reduces us to the previous case. 

The \teichmuller space $\mathcal{T}(S)$ is the space of marked conformal structures on $S$. Alternatively by uniformization, it is the space of complete marked hyperbolic metrics on $S$. The mapping class group $\textup{Mod}(S)$ is the group of orientation preserving diffeomorphisms of $S$ modulo isotopy. $\textup{Mod}(S)$ acts on $\mathcal{T}(S)$ by changing the marking and the quotient $\mathcal{M}(S) = \textup{Mod}(S) \backslash \mathcal{T}(S)$ is the moduli space of Riemann surfaces. 

The \teichmuller metric is given by 
\[
d_{\mathcal{T}}(X, Y) = \frac{1}{2} \inf_{f} \log K(f)
\]
where the infimum is taken over all quasi-conformal maps $f: X \to Y$, and $K(f)$ is the quasi-conformal constant of $f$. The group $\textup{Mod}(S)$ acts by isometries of the \teichmuller metric. For $\epsilon > 0$ small enough, the $\epsilon$-thin part $\mathcal{T}(S)_\epsilon$ is the set of hyperbolic surfaces $X$ that contain a simple closed curve with hyperbolic length less than $\epsilon$. The thin part $\mathcal{T}(S)_\epsilon$ is $\textup{Mod}(S)$ invariant. 

For a Riemann surface $X$, let $\mathcal{Q}(X)$ be the set of meromorphic quadratic differentials on $X$ with simple poles at the punctures. If $(k_1, k_2, \dots, k_r)$ are the multiplicities of the zeros then $k_1 + k_2 + \dots + k_r = 2g-2+n$, where $n$ is the number of punctures. A quadratic differential is equivalent to a half-translation structure on $S$, i.e. it defines charts from $S$ to $\C$ with transition functions of the form $\pm z + c$. The resulting flat metric has a cone singularity with cone angle $(k+2)\pi$ at a $k$-order zero (or with $k=-1$ for a simple pole) of the differential. A quadratic differential is unit area if the corresponding singular flat metric has area 1. The space of unit area quadratic differentials $\mathcal{Q}$ can be identified with the unit cotangent bundle to $\mathcal{T}(S)$ \cite{Hub-Mas}. The space $\mathcal{Q}$ is stratified according to the multiplicity of its zeros: we denote the strata with multiplicities $\alpha = (k_1, k_2, \dots, k_r)$ by $\mathcal{Q}(\alpha)$. A stratum $\mathcal{Q}(\alpha)$ may be disconnected. The number of connected components is finite and they have been classified. See \cite{Kon-Zor}, \cite{Lan}, \cite{Boi-Lan}. 

The periods/holonomies for a fixed basis of the homology of $S$ relative to the singularities give local co-ordinates on each stratum of quadratic differentials. The natural volume form in these co-ordinates, called the Masur-Veech measure denoted by $\mu_{\text{hol}}$, can be thought of as an analog of the Liouville measure. It is invariant under $\textup{Mod}(S)$ and descends to finite measure on the moduli space which we continue to denote $\mathcal{Q}(\alpha)$. See \cite{Mas}, \cite{Vee}.

In the flat metric defined by $q$, a {\em saddle connection} is a geodesic segment in the $q$-metric that connects a pair of (same or distinct) singularities. For a small enough $\epsilon > 0$, the $\epsilon$-thin part $\mathcal{Q}(\alpha)_\epsilon$ is the set of quadratic differentials $q \in \mathcal{Q}(\alpha)$ such that some saddle connection has $q$-length squared less than $\epsilon$. 

The affine action of $SL(2, \R)$ on $\C = \R^2$ preserves the transition functions to give a natural $SL(2, R)$ action on each stratum $\mathcal{Q}(\alpha)$. The action of the diagonal part defines the \teichmuller geodesic flow. The compact part $SO(2, \R)$ leaves the underlying conformal structure unchanged. Thus one gets an isometric embedding $\H^2 = SL(2, \R)/ SO(2, R) \rightarrow \T(S)$. These embeddings foliate $\T(S)$ and are called \teichmuller discs. We let $\D(q)$ be the \teichmuller disc given by the $SL(2, \R)$ orbit of $q$ which we denote by $SL(2, \R)(q)$.

The points $q'$ in $SL(2, \R)(q)$ for which a particular saddle connection $\beta$ has $q'$-length squared shorter than $\epsilon$ projects to a horoball in $\D(q)$. The point at infinity for $\D(q)$ is given by the direction in which $\beta$ is vertical. The ratios of holonomies of saddle connections that are parallel remain constant on $SL(2,\R)(q)$. Thus, in a collection of parallel saddle connections, the saddle connection with the shortest holonomy in $q$ determines the horoball.

\subsubsection*{$SL(2,\R)$ orbit closures and invariant measures.}
Recently, Eskin and Mirzakhani in \cite[Theorem 1.4]{Esk-Mir} show that ergodic $SL(2,\R)$-invariant probability measures are of Lebesgue class and supported on invariant complex submanifolds in $\mathcal{Q}(\alpha)$. These manifolds are affine in the sense that they are given by linear equations. Going further, Eskin, Mirzakhani and Mohammadi in \cite[Theorem 2.1]{Esk-Mir-Moh} show that all $SL(2, \R)$ orbit closures are affine invariant submanifolds. See \cite[Section 1]{Esk-Mir} for more details. More recently, Filip shows that these submanifolds are algebraic subvarieties. \cite{Fil}. 

Let $\mu$ be an ergodic $SL(2, \R)$-invariant probability measure supported on an affine invariant submanifold $\mathcal{N}$. For $\epsilon > 0$ small enough, the $\epsilon$-thin part $\mathcal{N}_\epsilon$ is the subset of $q \in \mathcal{N}$ such that some saddle connection has $q$-length squared less than $\epsilon$. Saddle connections $\beta_1, \beta_2$ are $\mathcal{N}$-parallel if they  are parallel for an open set of quadratic differentials in $\mathcal{N}$. We assume the following {\em regularity} condition for $\mu$. For $\epsilon, \kappa > 0$, let $\mathcal{N}_{\epsilon, \kappa}$ be the subset of $q \in \mathcal{N}$ that have at least one saddle connection $\beta_1$ with $\ell^2_q(\beta_1) \leqslant \epsilon$ and a saddle connection $\beta_2$ not $\mathcal{N}$-parallel to $\beta_1$ with $\ell^2_q(\beta_2) \leqslant \kappa$. We assume that there exists $m_1 > 0$ such that for $\epsilon, \kappa$ small enough
\[
\mu \left( \mathcal{N}_{\epsilon, \kappa} \right) \leqslant m_1 \epsilon \kappa.
\]
For $\mu_{\text{hol}}$, \cite[Section 10, Claim (7)]{Mas-Smi} proves the regularity above. A weaker regularity for any $SL(2,\R)$-invariant measure is proved in \cite[Theorem 1.2]{Avi-Mat-Yoc}. 

\subsubsection*{$SL(2,\R)$-invariant loci.} 
For $q \in \mathcal{N}$, let $V(q) \in \R^2 \setminus \{(0,0)\}$ be an assignment of a non-empty subset (with multiplicity) of holonomies of saddle connections on $q$. We require that the assignment varies linearly under $SL(2,\R)$ action, i.e. $V(g q) = g V(q)$ for all $g \in SL(2, \R)$. Such an assignment $V$ will be called a $SL(2,\R)$-invariant locus. Let $c(V,\mu)$ be the Siegel-Veech constant associated to $V$ and $\mu$. We assume $V$ is such that it satisfies $c(V,\mu) > 0$. 

For $R \geqslant 1$, the $\epsilon/R$-thin part of $\mathcal{N}$ corresponding to $V$ is the subset of $q \in \mathcal{N}$ for which some saddle connection with holonomy in $V(q)$ has $q$-length squared less than $\epsilon/R$.  We denote the set by $\mathcal{N}(V)_{\epsilon/R}$. The regularity condition and the Siegel-Veech formula \ref{Siegel-Veech} can be used to prove the volume asymptotic
\begin{equation}\label{n-vol}
\lim_{R \to \infty} \frac{\mu(\mathcal{N}(V)_{\epsilon/R})}{\pi \epsilon/R} = c(V, \mu).
\end{equation}
See \cite[Section 7]{Esk-Mas-Zor} for the main ideas.

Here, we consider excursions in the horoballs for saddle connections with holonomy in $V$. Let $E_V(\gamma, T)$ be the sum till time $T$ of excursions of $\gamma$ in horoballs for saddle connections with holonomy in $V$. Let $N_V$ be the number of such excursions of $\gamma$ till time $T$. The main theorem we prove is the following.

\begin{theorem}\label{saddles-V}
Let $\mu$-be a regular $SL(2,\R)$-invariant measure supported on an affine invariant submanifold $\mathcal{N}$. For $\mu$-almost every $q \in \mathcal{N}$, the \teichmuller geodesic ray $\gamma$ that $q$ determines satisfies
\[
\lim_{T \to \infty} \frac{E(\gamma, T) - \max\limits_{k \leqslant N_V} E(\gamma, H_k)}{T \log T} = 2 \epsilon c(V,\mu)
\]
where $c(V,\mu)$ is the Siegel-Veech associated to $V$ and $\mu$. 
\end{theorem}

\subsubsection{Configurations with cylinders.}
For the analog of Theorem \ref{wordmetric}, we state a special case of Theorem \ref{saddles-V}. For completeness we give some background.

For a connected component of a stratum $\mathcal{Q}(\alpha)$, a configuration $\cC$ of saddle connections is a geometric type of maximal collections of homologous saddle connections on a translation or half-translation surface in it. Here, the homology is the appropriate relative homology; see \cite{Mas-Zor} for details. The condition in homology implies that the saddle connections in a configuration are parallel. For holomorphic 1-forms their holonomies coincide. For quadratic differentials their holonomies can take two values which differ by a factor of 2. The saddle connections with the smaller holonomy will be called the small saddle connections in $\cC$. It was shown in \cite{Esk-Mas-Zor} and \cite{Mas-Zor} that in a $\mu_{\text{hol}}$-typical degeneration all saddle connections in some configuration shrink to length zero. A configuration $\cC$ gives a $SL(2,\R)$-invariant locus $V_{\cC}$ for $\mathcal{Q}(\alpha)$. By the discussion above $c(V, \mu_{\text{hol}}) > 0$.

A special subset of configurations corresponds to metric cylinders. A metric cylinder is an embedded cylinder that is a union of freely homotopic closed trajectories of $q$ such that the boundary components are a concatenation of saddle connections. If some of the saddle connections in a configuration $\cC$ bound a metric cylinder, we call $\cC$ a configuration with cylinders. Masur and Zorich \cite{Mas-Zor} show that each boundary component of such cylinders has exactly one or two saddle connections in $\cC$. The $q$-length of the core curve is equal to the boundary saddle connection or twice the length of one of the boundary saddle connections depending on the case.

Given $\cC$, the thin part of $SL(2, \R)(q)$ corresponding to $\cC$ are points for which the length squared of the small saddle connections in $\cC$ is less than $\epsilon$. Its projection to $\D(q)$ is a horoball. The point at infinity for the horoball is the direction in which the saddle connections in $\cC$ are vertical. 

Quadratic differentials on hyperbolic surfaces with short curves necessarily have short saddle connections but the converse need not be true. However, if $\epsilon$ is sufficiently small compared to the $q$-area of a cylinder then the core curve is also short in the underlying hyperbolic metric. For some constant $0 < \sigma < 1$ small enough depending on the orbit closure, we can specialize further to configurations with cylinders in which some cylinder has area at least $\sigma$. Such a restriction gives a horoball \texttt{"}packing\texttt{"}: any point in $\D(q)$ is contained at most $3g- 3 + n$ horoballs. By construction, the packing is $\textup{Mod}(S)$ equivariant. Masur \cite{Mas} showed that in each \teichmuller disc the packing satisfies Sullivan's criteria and used it to prove the lower bound in the log law: a Lebesgue typical geodesic ray in every \teichmuller disc is recurrent to the thick part with $\limsup$ of the maximum depth in $\T(S)_\epsilon$ asymptotically of size $(1/2) \log T$. 

Let $V$ be the subset of holonomies of saddle connections forming configurations with cylinders such that some cylinder has area at least $\sigma$. For a geodesic $\gamma \in \mathcal{Q}(\alpha)$, let $E_{\text{cyl}_\sigma}(\gamma,T)$ be the sum till time $T$ of excursions of $\gamma$ in horoballs for such configurations. Let $N_{\text{cyl}_\sigma}(T)$ be the number of such excursions till time $T$. As a special case of Theorem \ref{saddles-V}

\begin{theorem}\label{cyl}
For $\mu_{\text{hol}}$-almost every $q \in \mathcal{Q}(\alpha)$, the \teichmuller geodesic $\gamma$ that $q$ determines satisfies 
\[
\lim_{T \to \infty} \frac{E_{\text{cyl}_\sigma}(\gamma, T) - \max\limits_{k \leqslant N_{\text{cyl}_\sigma}} E(\gamma, H_k)}{T \log T} = 2 \epsilon c_{\text{cyl}_\sigma}(\alpha),
\]
where $c_{\text{cyl}_\sigma}(\alpha)$ is the Seigel-Veech constant for $\mathcal{Q}(\alpha)$ for configurations with cylinders such that some cylinder is of area at least $\sigma$.
\end{theorem}

Using Theorem \ref{cyl}, we can prove a lower bound on word-metric growth along typical \teichmuller geodesics. The key point is that along a \teichmuller geodesic the twisting in the core curve of a metric cylinder is up to a uniform multiplicative constant, $A/\epsilon$ times the excursion, where $A$ is the $q$-area of the cylinder. See \cite[Proposition 2.7]{Gad-Mah-Tio}. By Mumford compactness the quotient $\mathcal{M}(S) \setminus \mathcal{M}(S)_{\epsilon}$ is compact. Hence, $\textup{Mod}(S)$ is quasi-isometric to $\T(S) \setminus \T(S)_\epsilon$. With a basepoint $X_0$ in the thick part, the orbit $\textup{Mod}(S) X_0$ will be called a \teichmuller lattice. If $\gamma$ is recurrent to the thick part then along recurrence times $\gamma_T$, there is a lattice point $h_T X_0$ closes to $\gamma_T$. The distance between $\gamma_T$ and $h_T X_0$ is bounded by the diameter of $\mathcal{M}(S) \setminus \mathcal{M}(S)_\epsilon$. Because of the compactness of the thick part, this diameter is finite. As shown in \cite[Proposition 3.11]{Gad-Mah-Tio}, along a recurrent \teichmuller geodesic $\gamma$ the total excursion $E(\gamma, T)$ in the Masur collection gives a coarse lower bound on the word metric of the approximating group elements $h_T$, i.e. there exists constants $a_1, a_2 >  0$ such that 
\[
d_G(1, h_T) \geqslant a_1 E(\gamma, T) - a_2.
\]
Hence, as a direct consequence of Theorem \ref{cyl} we get
\begin{theorem}\label{word-metric-2}
There exists  a constant $M_2 > 0$ depending on the word metric such that for Leb-almost every $q \in \mathcal{Q}(X_0)$ the approximating group elements $h_T$ along the \teichmuller geodesic $\gamma$ that $q$ determines satisfy
\[
M_2T \log T < d_G(1, h_T).
\]
for all $T$ sufficiently large depending on $q$.
\end{theorem}

\subsection{Strategy of proof}

The key idea is to approximate the sum of excursions till time $T$ by an integral over time of a function defined over $T^1 X_{\text{cusps}}$ or $\mathcal{N}(V)_\epsilon$. This function is not $L^1$. Analyzing the largest excursion, we prove that if for some $c> 1/2$ it exceeds $T(\log T)^c$, then it is the unique excursion that exceeds this threshold. This follows from a Borel-Cantelli argument which requires quasi-independence of excursions. We use mixing of the geodesic flow to establish quasi-independence. See Proposition \ref{two-large}. By removing the largest excursion from the sum we get a quantity that can be approximated by a suitable $T$-dependent truncation of the above function. This truncation is $L^1$. The leading term of its $L^1$ norm is a constant times $\log T$. The constant is in terms of the proportional volume of the cusp neighborhoods. To conclude the proof of the main theorems, we apply an effective ergodic theorem to the truncation. This shows that he integral over $[0,T]$ of the truncation is equal to $T$ times the $L^1$ norm of the truncation with an error term which is $o(T \log T)$. To prove the effective ergodic theorem viz. Theorem \ref{ergodic-rate}, we use a specific decay of correlations for the geodesic flow. This decay of correlations is independently due to Moore and Ratner \cite{Moo} \cite{Rat} in the context of non-uniform lattices in $SL(2, \R)$. For quadratic differentials, this is due to Avila-Resende \cite{Avi-Res} for the Masur-Veech measure, and Avila-Gou\"ezel for general $SL(2,\R)$-invariant measures.  

In the quadratic differentials setting matters are complicated by the fact that a half-translation surface can have several non-homologous configurations of short saddle connections. While this number is finite for any given half-translation surface there is no upper bound for it over the $SL(2,\R)$-orbit closure. This means that a \teichmuller geodesic can do several excursions simultaneously and typically it does so. We impose a regularity assumption for the $SL(2,\R)$-invariant measure namely quasi-independence for two non-homologous configurations to be simultaneously short. Our main technical work leverages this quasi-independence and a bound due to Eskin and Masur \cite{Esk-Mas} for the number of short saddle connections in terms of length of shortest saddle connection, to prove that the truncation is indeed $L^1$. We also show that asymptotically the leading term of its $L^1$-norm is a constant times $\log T$. The constant is related to the asymptotic of volumes of thin parts. By Siegel-Veech theory, these are the associated Siegel-Veech constants.

\subsection{Acknowledgements} 

The central question considered in the paper arose in joint work with J. Maher and G. Tiozzo \cite{Gad-Mah-Tio}. I thank them for the initial discussion and for brining the paper by Diamond and Vaaler \cite{Dia-Vaa} to my attention. I thank J. Chaika, C. Matheus and M. Pollicott for useful discussions about the ergodic theory, J. Athreya and A. Eskin for useful discussions about the volume asymptotics for the thin parts of the moduli spaces of quadratic differentials. I thank the Institut Henri Poincar\'e and the Newton Institute for their hospitality during which parts of this work were done. I thank the Global Research Fellowship with the Institute of Advanced Study at the University of Warwick for the support. 
 
\section{Ergodic Theory}
This section develops the more abstract ergodic theoretic tools which will be used later. In particular, the main goal is to derive the effective ergodic theorem viz. Theorem \ref{ergodic-rate} which gives a uniform rate of convergence in the ergodic theorem simultaneously for a sequence of non-negative functions that satisfy a certain decay of correlations. 

Let $(X, \mathcal{B}, \ell)$ be a probability measure space. Let $g_t$ be a measure preserving flow on $X$ such that $g_t$ is exponentially mixing. More precisely, we assume that an appropriate subspace of $L^2(X)$ satisfies following decay of correlations: if $f_1$ and $f_2$ are functions in the subspace then $\int_X f_1 \, d\ell = \int_X f_2 \,d\ell= 0$ and there exists constants $K, \rho> 0$ such that
\begin{equation}\label{corr}
\left\vert \int_X f_1(g_s x) f_2(g_t x) \, d\ell \right\vert \leqslant K \vert t-s\vert e^{-\rho \vert t - s \vert} \Vert f_1 \Vert_{L^2} \Vert f_2 \Vert_{L^2}.
\end{equation}

For a function $f \in L^1(X)$, let $I(f) = \int_X f d\ell$. We denote by $\mathcal{L}$ the subspace in $L^2(X)$ of functions $f$ such that  the function $f - I(f)$ satisfies the decay of correlations \ref{corr} above.
\begin{lemma}\label{L2}
Any function $f \in \mathcal{L}$ with $I(f) = 0$ satisfies: 
\begin{equation}\label{L2-bound}
\int_X \left(\int_0^T f(g_t x) \, dt\right)^2 \, d\ell \leqslant 2KT \Vert f \Vert^2_{L^2}. 
\end{equation}
\end{lemma}

\begin{proof}
Observe that
\begin{eqnarray*}
\int_X \left(\int_0^T f(g_t x) \, dt\right)^2 d\ell &=& \int_X \left( \int_0^T \int_0^T f(g_s x) f(g_t x) \, ds \,  dt \right) d\ell \\
&=& \int_0^T \int_0^T \left(\int_X f(g_s x) f(g_t x) d\ell \right) ds \, dt\\
&\leqslant& \int_0^T \int_0^T K\vert t-s \vert e^{-\rho \vert t -s\vert} \Vert f \Vert^2_{L^2} \, ds \, dt
\end{eqnarray*} 
where we have used the decay of correlations \ref{corr} in the last inequality. A direct computation shows
\begin{eqnarray*}
\int_0^T \int_0^T K\vert t-s \vert e^{-\rho\vert t -s\vert} \Vert f \Vert^2_{L^2} \, ds \, dt &=& K\Vert f \Vert^2_{L^2} \left( \frac{T}{\rho^2}(1+e^{-\rho T}) + \frac{2}{\rho^3}( -1+ e^{-\rho T})\right) \\
&\leqslant& 2KT \Vert f \Vert^2_{L^2}
\end{eqnarray*}
finishing the proof of the lemma.
\end{proof}

Suppose $n: \R \to \N$ is a function that is constant on each interval $[2^k, 2^{k+1})$. 
\begin{theorem}\label{ergodic-rate} 
For any $c>1/2, m > 1$ and any sequence of non-negative functions $f_j \in \mathcal{L}$, almost every $x$ satisfies
\begin{equation*}
\begin{split}
\frac{1}{m} T \Vert f_n \Vert_{L^1}- T^{1/2}(\log T)^c  \left(\Vert f_n \Vert^2_{L^2} - \Vert f_n \Vert^2_{L^1} \right)^{1/2} &\leqslant \int_0^T f_n(g_t x) dt \\
&\leqslant mT\Vert f_n \Vert_{L^1} - T^{1/2}(\log T)^c \left(\Vert f_n \Vert^2_{L^2} - \Vert f_n \Vert^2_{L^1}\right)^{1/2}
\end{split}
\end{equation*}
for all $T$ large enough depending on $x$ and where $n = n(T)$. 
\end{theorem}

\begin{proof}
Given a function $f \in \mathcal{L}$ define
\[
F(x) = f(x) - I(f).
\]
Then $I(F) = 0$ and so by lemma \ref{L2}
\[
\int_X \left( \int_0^T F(g_t x) \,dt \right)^2 d\ell \leqslant 2KT \Vert F \Vert^2_{L^2}
\]
for all $T$. By Chebysheff's inequality, for any positive function $r(T, F)$ we have 
\begin{equation}\label{cheby1}
\ell \left( x \text{ such that } \left(\int_0^T F(g_t x) \,dt \right)^2 \geqslant r(T, F) \right) \leqslant \frac{2KT\Vert F \Vert^2_{L^2}}{r(T, F)}.
\end{equation}
Let $c > 1/2$ and set $r(T, F) = T(\log T)^{2c}\Vert F \Vert^2_{L^2}$ in \ref{cheby1}. Then we get
\begin{equation}\label{cheby2}
\ell \left( x \text{ such that } \left(\int_0^T F(g_t x) \,dt\right)^2 \geqslant T (\log T)^{2c}\Vert F \Vert^2_{L^2} \right) \leqslant \frac{2K}{ (\log T)^{2c}}.
\end{equation}
Starting from our sequence $f_j$, let $F_j$ be the sequence of functions given by 
\[
F_j(x) = f_j(x) - I(f_j).
\] 
The estimate \ref{cheby2} above is satisfied by all functions $F_j$ and in particular by $F_n$ where $n = n(T)$. Fix $r = 1/a$ for some positive integer $a>1$. Observe that for the sequence $T_k = 2^{rk}$ 
\[
\sum_{k=1}^{\infty} \frac{2K}{(\log T_k)^{2c}} = \sum_{k=1}^{\infty} \frac{2K}{(rk)^{2c}} < \infty.
\]
Hence by Borel-Cantelli lemma, almost every $x$ satisfies 
\begin{equation}\label{k}
\left(\int_0^{T_k} F_n(g_t x) \,dt\right)^2  \leqslant T_k (\log T_k)^{2c}\Vert F_n \Vert^2_{L^2}
\end{equation}
for all $k$ large enough depending on $x$. Similarly, setting $r(T, F) = (T/2^r)(\log (T/2^r))^{2c} \Vert F \Vert^2_{L^2}$ and shifting $n(T)$ to $n(T/2)$, the same reasoning by Borel-Cantellii lemma implies that almost every $x$ satisfies 
\begin{equation}\label{k+1}
\left( \int_0^{T_{k+1}} F_n(g_t x) \, dt \right)^2 \leqslant T_k(\log T_k )^{2c} \Vert F_n \Vert^2_{L^2}
\end{equation}
for all $k$ large enough depending on $x$. Hence a full measure set of $x$ satisfy both \ref{k} and \ref{k+1}. Noting that $F_n(g_t x) = f_n(g_t x) - I(f_n)$ and $\Vert F_n \Vert^2_{L^2} = \Vert f_n \Vert^2_{L^2} -I(f_n)^2$, the above estimates can be rewritten as
\[
\left \vert \int_0^{T_k} f_n(g_t x) \,dt  - T_k I(f_n) \right\vert \leqslant T_k^{1/2}(\log T_k)^c \left(\Vert f_n \Vert_{L^2}- I(f_n)^2\right)^{1/2}
\]
and
\[
\left \vert \int_0^{T_{k+1}} f_n(g_t x) \,dt  - T_{k+1} I(f_n) \right\vert \leqslant T_k^{1/2}(\log T_k)^c \left(\Vert f_n \Vert_{L^2}- I(f_n)^2\right)^{1/2}.
\]
Over the intermediate times $T_k < T < T_{k+1}$ the number $n$ does not vary. So the function $f_n$ being considered remains the same. Now we use the assumption that $f_n$ is a non-negative function to get an estimate such as above for these intermediate times. Since $f_n$ is non-negative, the time integral of $f_n$ is non-decreasing. In particular,
\[
\int_0^{T_k} f_n(g_t x) \,dt \leqslant \int_0^T f_n(g_t x) \,dt \leqslant \int_0^{T_{k+1}} f_n(g_t x) \,dt. 
\]
Observe that
\begin{equation}\label{left}
T_k I(f_n) -  T_k^{1/2}(\log T_k)^c \left(\Vert f_n \Vert^2_{L^2} - I(f_n)^2\right)^{1/2} \geqslant \frac{1}{2^r} TI(f_n) - T^{1/2} (\log T)^c \left( \Vert f_n \Vert^2_{L^2} - I(f_n) \right)^{1/2} 
\end{equation}
and 
\begin{equation}\label{right}
T_{k+1} I(f_n) + T_k^{1/2}( \log T_k)^c \left( \Vert f_n \Vert^2_{L^2} - I(f_n) \right)^{1/2} \leqslant 2^rTI(f_n) + T^{1/2} (\log T)^c \left( \Vert f_n \Vert^2_{L^2} - I(f_n) \right)^{1/2}. 
\end{equation}
Finally, note $I(f_n) = \Vert f_n \Vert_{L^1}$. The left hand side of \ref{left} is the lower bound in \ref{k} and the left hand side in \ref{right} is the upper bound in \ref{k+1}. Thus we get
\begin{align*}
\frac{1}{2^r} TI(f_n) - T^{1/2} (\log T)^c \left( \Vert f_n \Vert^2_{L^2} - I(f_n) \right)^{1/2} &\leqslant \int_0^T f_n(g_t x) \, dt \\ &\leqslant 2^rTI(f_n) + T^{1/2} (\log T)^c \left( \Vert f_n \Vert^2_{L^2} - I(f_n) \right)^{1/2}.
\end{align*}
The theorem follows by choosing $a$ large enough such that $r = 1/a$ satisfies $2^r < m$. 
\end{proof}

We also prove a variant of Lemma \ref{L2} which we will need later for quasi-independence of excursions.

\begin{lemma}\label{quasi-ind}
For any $S_1 < S_2 < T$ and non-negative function $f \in \mathcal{L}$
\begin{equation}\label{large-bound}
\int_X  \left( \int_{S_1}^{S_2} f(g_s x) \,ds \int_{S_2}^T  f(g_t x) \,dt \right)d\ell  < (S_2-S_1)(T-S_2) \Vert f \Vert^2_{L^1} + \frac{5K}{\rho} \left( \Vert f \Vert^2_{L^2} - I(f)^2\right)
\end{equation}
where $K, c$ are the constants in the decay of correlations \ref{corr}.
\end{lemma}

\begin{proof}
For any function $f \in \mathcal{L}$, define $F$ by
\[
F(x) = f(x) - I(f).
\]
Then $I(F) =0$ and so it satisfies the decay of correlations \ref{corr}. Note that
\[
\int_X F(g_s x) F(g_t x) \,d\ell = \int_X f(g_s x) f(g_t x) \,d\ell - I(f)^2
\]
and $\Vert F \vert^2_{L^2} = \Vert f \Vert^2_{L^2} - I(f)^2$. It follows that $f$ satisfies
\[
\left\vert \int_X f(g_s x) f(g_t x) \,d\ell - I(f)^2 \right\vert \leqslant K\vert t-s\vert e^{-\rho \vert t- s\vert} \left( \Vert f \Vert^2_{L^2} - I(f)^2\right)
\]
which implies
\[
\left \vert \int_X f(g_s x)  f(g_tx) \,d\ell \right\vert \leqslant I(f)^2 + K\vert t-s\vert e^{-\rho \vert t- s\vert} \left( \Vert f \Vert^2_{L^2} - I(f)^2\right).
\]
For non-negative functions this implies
\begin{eqnarray*}
\int_X  \left( \int_{S_1}^{S_2} f(g_s x)\,ds \int_{S_2}^T  f(g_t x) \,dt \right)d\ell  &=& \int_{S_2}^T \int_{S_1}^{S_2} \left( \int_X f(g_s x) f(g_t x) \,d\ell \right) ds \,dt \\
&\leqslant& \int_{S_2}^T \int_{S_1}^{S_2} \left[ I(f)^2 + K\vert t-s\vert e^{-\rho \vert t- s\vert} \left( \Vert f \Vert^2_{L^2} - I(f)^2\right) \right] ds \,dt\\
&<& (S_2-S_1)(T-S_2) I(f)^2 + \frac{5K}{\rho} \left( \Vert f \Vert^2_{L^2} - I(f)^2\right)
\end{eqnarray*}
where the last inequality follows from a direct computation.
\end{proof} 

\section{Partial sums of excursions for non-uniform lattices in $SL(2, \R)$}

The Liouville measure $\ell$ on $T^1\H^2$ is invariant under the $SL(2, \R)$ action and descends to a flow-invariant measure on $T^1 X = G \backslash T^1 \H^2$. To get a probability measure $\ell$ on $T^1X$ we normalize by passing to
\[
d\ell \rightarrow \frac{1}{2\pi \vert \chi(X) \vert} d\ell.
\]
For notational simplicity we continue to call the probability measure $d\ell$. 

The geodesic flow on $T^1 X$ is given by the action of the diagonal subgroup of $SL(2, \R)$. By a classical result due to Hopf, the geodesic flow on $X = G \backslash \H^2$ is ergodic with respect to $\ell$. In fact, it is known to be exponentially mixing. As shown in \cite{Moo}, $SO(2, \R)$-invariant $L^2$-functions on $T^1 X$ satisfy the following decay of correlations for the diagonal flow: There exists constants $K> 0, \rho> 0$ such that for any pair $f_1, f_2$ of $SO(2, \R)$-invariant $L^2$-functions on $T^1X$ with $\int_{T^1X} f_1 d\ell = \int_{T^1X} f_2 d\ell = 0$ 
\begin{equation}\label{Ratner-decay}
\int_{T^1X} f_1(x) f_2(g_t x) \,d\ell \leqslant  K t e^{-\rho t} \Vert f_1 \Vert_{L^2} \Vert f_2 \Vert_{L^2}.
\end{equation}
See also  \cite[Theorem 2]{Rat}, \cite[Corollary 2.1]{Mat}. In particular, the lifts to $T^1X$ of $L^2$-functions on $X$ are by default $SO(2, \R)$ invariant. So the above decay of correlations applies to them. 

For $R \geqslant 1$, let $Y_R$ be the subset of the horoballs $\mathcal{H}$ consisting of those points which are at least distance $\log R$ from the boundary of the horoballs in the hyperbolic metric, i.e.
\[
Y_R := \bigcup_{H \in \mathcal{H}} \{ x \in H \ : \ d(x, \partial H) \geqslant \log R \}. 
\]
Let $X_R \subset X $ be the quotient of $G \backslash Y_R$. In particular, $X_1 = X_{\text{cusp}}$. We will write $T^1 Y$ for the restriction of the unit tangent bundle to any subset $Y \subset X$. An elementary calculation in hyperbolic space shows that 
\[
\ell(T^1X_R) = \frac{1}{R} \frac{\ell(T^1 X_{\text{cusp}})}{\ell(T^1 X)} = \frac{C_X}{R}
\] 
where to simplify notation, henceforth we will denote $\ell(T^1X_{\text{cusp}}) /  \ell(T^1 X)$ by $C_X$. Let $\chi_R$ be the characteristic function of $T^1 X_R$ and let $\phi_R = \chi_{R/2} - \chi_R$. Note that $\Vert \phi_R \Vert_{L^1} = C_X/ R$ and $\Vert \phi_R \Vert_{L^2} = \sqrt{C_X}/\sqrt{R}$. During an excursion of size at least $R$, a geodesic $\gamma$ must cross $T^1 X_{R/2}  \setminus T^1 X_R$ twice during a complete excursion and at least once during a partial excursion. By basic hyperbolic geometry, the geodesic spends time greater than $\log 2$ each time it crosses $T^1 X_{R/2} \setminus T^1 X_R$.  

The next proposition allows us to show that along Leb-almost every geodesic ray, for all times $T$ large enough there is at most a single \texttt{"}large\texttt{"} excursion. The proposition is a continuous time refinement of \cite[Lemma 2.1]{Dia-Vaa} and the proof uses Lemma \ref{quasi-ind}. 

\begin{proposition}\label{two-large}
For any $c > 1/2$ and for $\ell$-every $v \in T^1 X$ there exists $T(v)$ such that for all $T < T(v)$
\[
E(\gamma, H_i) \geqslant T (\log T)^c
\]
for at most a single $H_i \in H_{\gamma, T}$ and where $\gamma$ is the geodesic ray with $v(\gamma_0) = v$.
\end{proposition}

\noindent For the rest of the discussion, let $T_n = 2^n$. Proposition \ref{two-large} follows from the following proposition.

\begin{proposition}\label{discrete-times}
For any $c > 1/2$ and for $\ell$-every $v \in T^1 X$ there exists non-negative integer $n(v)$ such that for all $T_n > T_{n(v)}$ 
\[
E(\gamma, H_i) \geqslant T_{n-1} (\log T_{n-1})^c
\]
for at most single $H_i \in H_{\gamma, T_n}$ and where $\gamma$ is the geodesic ray with $v(\gamma_0) = v$. 
\end{proposition}

\begin{proof}
Let $\lambda = \log T_{n-1} + c \log \log T_{n-1}$. By basic hyperbolic geometry, the time a geodesic takes to go from the boundary of a horoball to $X_R$ where $R = T_{n-1} ( \log T_{n-1})^c$ is bounded between $\lambda$ and $\log \left( T_{n-1}(\log T_{n-1})^c + \sqrt{T_{n-1}^2 (\log T_{n-1})^{2c} - 1} \right) < \lambda + \log 2$. Similarly, let $\lambda' = \lambda - \log 2$. Then $\lambda'$ is a lower bound on the time it takes a geodesic to go from the boundary of a horoball to $X_{R/2}$ where $R = T_{n-1}( \log T_{n-1})^c$.
 
For positive integers $j \leqslant \lfloor T_n/ \lambda \rfloor$, let $S_j = j \lambda$. Let
\[
V_{n, k} = \left\{ v \in T^1X \text{ such that } \int_{S_k}^{S_{k+2} + \lambda'} \phi_R(\gamma(v_s)) \, ds > \log 2 \text{ and } \int_{S_{k+2}+ \lambda'}^{T_n} \phi_R(\gamma(v_t)) \, dt > \log 2 \right\}
\]
By applying Chebysheff's inequality to the estimate in Lemma \ref{quasi-ind} for the function $\phi_R$ we get
\begin{eqnarray*}
\ell(V_{n,k}) &\leqslant&  \frac{(S_{k+2} + \lambda' - S_K)(T_n - S_{k+2} - \lambda')}{(\log 2)^2}\frac{C_X^2}{R^2}\\ 
&+& \frac{1}{(\log 2)^2} \left(\frac{5KC_X}{\rho R} \left[ 1  - \frac{C_X}{R}\right] \right)\\
&<& \frac{6 \lambda C_X^2}{(\log 2)^2 T_{n-1}(\log T_{n-1})^{2c}} + \frac{5KC_X}{\rho (\log 2)^2 T_{n-1} (\log T_{n-1})^c}\\
&<& \frac{b_1}{T_{n-1} (\log T_{n-1})^{2c}} + \frac{b_2}{T_{n-1}(\log T_{n-1})^c}
\end{eqnarray*}
for some constants $b_1, b_2> 0$. In the second to last inequality we have used $T_n - S_k < T_n = 2 T_{n-1}$. 

Let $W$ be the set of  $v \in T^1X$ such that the corresponding geodesic $\gamma$ has two excursions $E(\gamma, H_i)$ and $E(\gamma, H_j)$ till time $T_n$ satisfying $E(\gamma, H_i) \geqslant T_{n-1}(\log T_{n-1})^c$ and $E(\gamma, H_j) \geqslant T_{n-1}(\log T_{n-1})^c$. Let $S$ be the time at which the first big excursion $E(\gamma, H_i)$ begins. Let $k$ be such that $S_k \leqslant S < S_{k+1}$. Because of our choice of $\lambda$ it follows that second big excursion $E(\gamma, H_j)$ cannot begin before $S_{k+2}$. Then because of the choice of $\lambda'$ the geodesic can not cross $T^1 X_{R/2} \setminus T^1 X_R$ during the second big excursion before $S_{k+2}+ \lambda'$. This means that $v \in V_{n,k}$. 
Let 
\[
V_n = \bigcup_{k=0}^{\lfloor T_n /\lambda \rfloor - 2} V_{n,k}
\]
Using the estimate on $\ell(V_{n,k})$ we get
\[
\ell(V_n) \leqslant  \sum_{k=1}^{\lfloor T_n/ \lambda \rfloor} \ell(V_{n,k}) < \frac{b_1 T_n}{T_{n-1} (\log T_{n-1})^{2c}} + \frac{b_2 T_n}{T_{n-1} (\log T_{n-1})^c} \leqslant \frac{2b_1}{(\log T_{n-1})^{2c}} + \frac{2b_2}{(\log T_{n-1})^c}.
\]
Since $c > 1/2$ it follows that 
\[
\sum_n \ell(V_n) < \infty
\]
Proposition \ref{discrete-times} then follows by the Borel-Cantelli lemma.
\end{proof}

\begin{proof}[Proof of Proposition \ref{two-large}]
Let $n$ be such that $T_{n-1} < T \leqslant T_n$. Since $E(\gamma, H) \geqslant T(\log T)^c$ implies $E(\gamma, H) > T_{n-1}(\log T_{n-1})^c$, Proposition \ref{two-large} follows from Proposition \ref{discrete-times}.
\end{proof}

\begin{remark}\label{tail}
It is important to observe that Proposition \ref{discrete-times} holds under the weaker condition that there is a constant $A> 1$ such that for $R$ sufficiently large
\[
\frac{1}{A} < \ell(T^1 X_R)  < A
\]
This observation will be of relevance for a similar proposition in the setting of quadratic differentials.
\end{remark}
Recall that $x_0$ is a base-point and $T^1_{x_0} X$ can be identified with $S^1$. 
\begin{corollary}\label{generic-S1}
For any $c > 1/2$ and Leb-almost every $r \in S^1$ there is $T(r)$ such that if $T > T(r)$ then 
\[
E(\gamma, H_i) \geqslant T(\log T)^c
\]
for at most single $H_i \in H_{\gamma, T}$ and where $\gamma$ is the geodesic ray from $x_0$ to $r$.
\end{corollary}

\begin{proof}
It follows from Proposition \ref{two-large} that the corollary is true for generic base-points. Suppose $\gamma_0$ and $\gamma_1$ are geodesic rays from distinct base-points $x_0$ and $x_1$ converging to the same point $r$ at infinity and let $H$ be a horoball. Let $\pi_H$ be the closest point projection to $H$ and let $a = d_{\partial H}(\pi_H x_0, \pi_H x_1)$.
Then we have the crude bound
\[
E(\gamma_0, H) - 2ae^{-\tau}-2 \leqslant E(\gamma_1, H) \leqslant E(\gamma, H) + 2ae^{-\tau}+2
\]
where $\tau$ is the minimum of $d(x_0, H)$ and $d(x_1, H)$. So for $H$ that is far enough the excursions by $\gamma_0$ and $\gamma_1$ are the same up to a uniform additive constant. This implies the corollary.
\end{proof}

Define the function $\psi: X \to \R$ by 
\[
\psi(x) = \left\{ \begin{array}{cc} 0 & \text{ if } x \in X_{\text{thick}}  \\ \left(\frac{2}{\pi} \right)e^{d(x, \partial X_{\text{thick}})} & \text{otherwise.} \end{array}
\right. 
\]
Let $\Psi : T^1 X \to \R$ be the lift of $\psi$ to $T^1X$. By default, the function $\Psi$ is $SO(2, \R)$-invariant.

Suppose a geodesic ray $\gamma$ has a complete excursion in a horoball $H$ entering and exiting $H$ at times $T_1$ and $T_2$ respectively then it follows from basic hyperbolic geometry that
\begin{equation}\label{complete-exc}
E(\gamma, H) - 2 <  \int_{T_1}^{T_2} \Psi(v(\gamma_t)) \,dt < E(\gamma, H) + 2
\end{equation}
i.e. for complete excursions the time integral of $\Psi$ is $E(\gamma, H)$ up to a uniform additive error. For the partial excursion, we have the crude estimate
\begin{equation}\label{partial-exc}
E(\gamma, H) - 2 < 2 \int_{T_1}^{T_2} \Psi(v(\gamma_t)) \,dt 
\end{equation}
which will prove sufficient for the purpose. 

Recall that for $\chi_R$ is the characteristic function of $T^1 X_R$. We define truncations of $\Psi$ by 
\[
\Psi_R(v)  = \Psi(v)( \chi_1(v) -  \chi_R(v) )
\]
where $\chi_1$ is the characteristic function of $T^1X_1 = T^1 X_{\text{cusp}}$.  Note that while $\Psi$ is not $L^1$ the truncations $\Psi_R$ satisfy 
\[
\Vert \Psi_R \Vert_{L^1}  = \frac{2C_X}{\pi} \log R \hskip 12pt \text{and} \hskip 10pt \Vert \Psi_R \Vert_{L^2} =  \frac{2\sqrt{C_X}}{\pi} \sqrt{R}. 
\]
Inequalities \ref{complete-exc} and \ref{partial-exc} show that partial sums of excursions (minus the largest excursion) i.e., the numerator in Theorem \ref{Esum} is estimated by the time integral of a suitable truncation of $\Psi$ up to an additive error that is linear in the number $N$ of excursions.  The next lemma shows that $N$ grows linearly in $T$.

\begin{lemma}\label{n-T-lemma}
There is a constant $q> 0$ such that for Leb-almost every $r \in S^1$
\[
\lim_{T \to \infty} \frac{N}{T} = \eta.
\]
\end{lemma}

\begin{proof}[Proof of Lemma \ref{n-T-lemma}]
The lemma follows from an approach similar to Schmidt's theorem in the theory of Diophantine approximation \cite[Theorem 1.1 with $k=1$]{Ath-Par-Tse}. Alternatively, we give a weaker but direct proof below.

For $v \in T^1(X \setminus X_{\text{thick}})$, let $\gamma$ be the geodesic such that $v(\gamma_0) = v$ i.e., the geodesic whose unit tangent vector at $t=0$ is $v$. For almost every $v$, the geodesic $\gamma$ intersects $\partial X_{\text{thick}}$ both in the forward and backward directions. Let $T_b< 0$ and $T_f> 0$ be the first instances of these intersections, i.e.the first instances backward and forward along $\gamma$ when it intersects the boundary of the horoball containing $\pi(v)$. Set 
\[
\xi(v) = \frac{1}{T_f - T_b}. 
\]
This defines a non-negative function $\xi: T^1(X \setminus X_{\text{thick}}) \to \R_{\geqslant 0}$ which we extend by setting it zero outside. It is straightforward to see that for almost every $r \in S^1$
\[
\int_0^T \xi(v(\gamma_t)) \,dt = N.
\]
We claim that $\xi \in L^1(T^1X)$. Let $R_k = 1 + (1/2^k)$ and consider $T^1 X_{R_k} \setminus T^1 X_{R_{k-1}}$. Since $\log (1 + (1/2^k)) = (1/2^k) - (1/2^{2k+1}) + \text{higher order terms}$ there exists a constant $b_3 > 0$ such that
\[
\ell \left(T^1 X_{R_{k-1}} \setminus T^1 X_{R_k} \right) < \frac{b_3}{2^k}
\]
for all $k$ large enough. By basic hyperbolic geometry, for any $v \in T^1 X_{R_k} \setminus T^1 X_{R_{k-1}}$
\begin{align*}
\xi(v) &\leqslant \frac{1}{2 \log \left(1+ (1/2^k) + \sqrt{(1+ (1/2^k))^2 - 1} \right)} \\
&< \frac{1}{2 \log \left( 1+ \sqrt{1/2^k} + (1/2^k)\right)  }\\
&< b_4 2^{k/2}
\end{align*}
for some constant $b_4 > 0$ and all $k$ large enough. This gives the bound
\[
\int_{T^1 X_{R_k} \setminus T^1 X_{R_{k-1}}} \xi \, d\ell < \frac{b_3 b_4}{2^{k/2}}.
\]
which proves the claim that $\xi \in L^1$. The lemma follows by applying the ergodic theorem to $\xi$.
\end{proof}

\begin{proof}[Proof of Theorem \ref{Esum}]
Consider the sequence of functions $\Psi_{2^k}$. While it is not necessary, for notational simplicity we set the constants $c> 1/2$ in Proposition \ref{two-large} and in Theorem \ref{ergodic-rate} to be equal. For $T$ such that $2^k \leqslant T < 2^{k+1}$ we set $n(T) = \lfloor k + c \log_2 k \rfloor$ where $\lfloor \hskip 3pt \rfloor$ is the greatest integer function. By Theorem \ref{ergodic-rate} applied to the sequence of functions $\Psi_{2^n}$, we have that for any $c> 1/2$ and $m>1$, $\ell$-almost every $v \in T^1 X$ satisfies
\[
\left\vert \int_0^T \Psi_{2^n} (v(\gamma_t)) \,dt - \frac{1}{m} T \Vert \Psi_{2^n} \Vert_{L^1} \right\vert \leqslant T^{1/2} (\log T)^{c_1}(\Vert \Psi_{2^n} \Vert^2_{L^2} - \Vert \Psi_{2^n} \Vert^2_{L^1})^{1/2}
\]
for all $T$ large enough depending on $v$. Let $r$ be the ratio $T(\log T)^c/ 2^n$. Then $1 < r < 3$. Substituting the $L^1$ and $L^2$ norms of $\Psi_{2^n}$ we see that the left hand side becomes:
\[
\left\vert \int_0^T \Psi_{2^n}(v(\gamma_t)) \,dt  - \frac{2C_X}{m\pi} T( \log T + c \log \log T - \log r) \right\vert  
\]
and the right hand side becomes
\[
T^{1/2}(\log T)^c \left( \frac{4C_X}{r \pi^2} T (\log T)^c - \frac{4C_X^2}{ \pi^2} (\log T + c \log \log T - \log r)^2\right)^{1/2} \leqslant \frac{2 \sqrt{C_X}}{\pi \sqrt{r}} T(\log T)^{3c/2}.
\]
Thus, we get
\begin{equation}\label{m-bound}
\left\vert \int_0^T \Psi_{2^n}(v(\gamma_t)) \,dt  - \frac{2C_X}{m\pi} T( \log T + c \log \log T - \log r) \right\vert  \leqslant \frac{2 \sqrt{C_X}}{\pi \sqrt{r}} T(\log T)^{3c/2}.
\end{equation}
We choose $c< 2/3$. Let $U_m$ be the full measure set in $T^1 X$ satisfying \ref{m-bound}. Consider the countable intersection
 \[
 U = \bigcap_{a \in \mathbb{N}}  U_{1+1/a}.
 \]
Then $U$ has full measure and for $v$ in $U$ the constraint $c < 2/3$ implies
\begin{equation}\label{main-lim}
\lim_{T \to \infty} \frac{1}{T \log T} \int_0^T \Psi_{2^n}(v(\gamma_t)) \,dt  = \frac{2C_X}{\pi}. 
\end{equation}
By the same reasoning as in the proof of Corollary \ref{generic-S1} the above limit is true for any base-point $x_0$ and Leb-almost every $r \in S^1$.

It remains to relate the time integral of $\Psi_{2^n}$ to partial sum of excursions. Enumerate the horoballs in $\mathcal{H}_{\gamma,T}$ as $H_1, \dots, H_N$ in the order of increasing time. In accordance with Corollary \ref{generic-S1}, suppose there is a single excursion
$E(\gamma, H_i) > T (\log T)^c$. Let $T_1 < T_2$ be the entry and exit times in $H_i$. Then notice that
\[
\int_{T_1}^{T_2} \Psi_{2^n} (v(\gamma_t)) \,dt \leqslant 2 T(\log T)^c.
\]
If there is partial excursion then let $T_3 < T$ be the time at which $\gamma$ enters $H_N$ and notice that
\[ 
\int_{T_3}^T \Psi_{2^n}(v(\gamma_t)) \,dt \leqslant 2 T(\log T)^c.
\]
Using the estimates above and also \ref{complete-exc} and \ref{partial-exc} we get
\begin{equation}\label{time-sum}
\begin{split}
\int_0^T \Psi_{2^n}(v(\gamma_t)) \,dt - 2N - 4T(\log T)^c &\leqslant  E(\gamma, T) - \max_{1\leqslant k \leqslant N} E(\gamma, H_k)\\
&\leqslant \int_0^T \Psi_{2^n}(v(\gamma_t)) \,dt + 2N + 2 T(\log T)^c.
\end{split}
\end{equation}
Theorem \ref{Esum} then follows from putting together \ref{main-lim}, \ref{time-sum} and Lemma \ref{n-T-lemma}. 
\end{proof}

\begin{proof}[Proof of Diamond-Vaaler theorem \ref{Diam-Vaal}]
For the modular surface $X = SL(2, \Z) \backslash \H^2$ the lift to $\H^2$ of the largest embedded cusp neighborhood in $X$ is the well-known Ford packing: in the upper half space model, we get horoballs resting at rational points, the Euclidean radius of the horoball with the point at infinity $p/q$ in reduced form being $1/2q^2$. 

With the cusp neighborhood fixed as above, Theorem \ref{Esum} for $X$ states that for any base-point and Leb-almost every $r \in S^1$
\begin{equation}\label{modular}
\lim_{T \to \infty} \frac{E(\gamma, T) - \max\limits_{k\leqslant N} E(\gamma, H_k)}{T \log T} = \frac{6}{\pi^2}.
\end{equation}
where $\gamma$ is the geodesic ray from some base-point $x_0$ to $r$. To derive the Diamond-Vaaler result (Theorem \ref{Diam-Vaal}) from the above limit, we relate excursions to continued fraction coefficients of $r$ and time $T$ along the geodesic to the number $n$ of continued fraction coefficients. 

In the upper half space model, for $r \in [0,1]$ irrational, let $[a_1, \dotsb, a_n, \dotsb ]$ be the infinite continued fraction expansion of $r$. Let $p_n/q_n = [a_1, \dots, a_n]$ be the $n$-th {\em convergent} of $r$ and let $H'_n$ be the horoball with $p_n/q_n$ as the point at infinity. We first consider vertical geodesics: for $r \in [0,1]$ let $\gamma'$ be the vertical geodesic ray from $(r, i) \in \H^2$ to $(r, 0) \in S^1 = \R \cup \infty$.

The ray $\gamma'$ has excursions in horoballs that are given by rational approximations of $r$ satisfying $\vert r - p/q \vert \leqslant 1/2q^2$. By a classical theorem for continued fractions, such rationals are a subset of the convergents $p_n/q_n$. If $a_n \geqslant 2$ then $a_n - 1 < E(\gamma', H'_n) < a_n+1$. However, if $a_n = 1$ then $\gamma'$ may or may not intersect $H'_n$ and we set $E(\gamma', H'_n) = 0$ if it does not. In any case, excursions of $\gamma'$ are equal to the coefficients up to a uniform additive error and hence we get 
\begin{equation}\label{coeff-sum}
\sum_{k = 1}^n E(\gamma', H'_k)- n \leqslant \sum_{k = 1}^n a_k  \leqslant   \sum_{k = 1}^n E(\gamma', H'_k) + n.
\end{equation}

By classical theory of continued fractions (\cite[Proposition 4.8.2(4)]{Bri-Stu}) for Leb-almost every $r$
\[
\lim_{n \to \infty} \frac{\log q_n}{n} = \frac{\pi^2}{12 \log 2}.
\]
Since $p_n/q_n \to r$, the same limit is true for $\log p_n/n$. Up to a transposition of columns, the matrix $Q_n$ with columns $[p_{n-1}, q_{n-1}]^t$ and $[p_n, q_n]^t$ is in $SL(2,\Z)$. The hyperbolic translation length of the matrix up to a uniform additive error is $2\log(\text{trace})$. By the above discussion $\log(\text{trace})$ is $\log q_n$ up to a uniform additive error. So let $2\log q_n = T_n$. Recall that $Q_n$ acts on the upper half plane by Mobius transformations. Geometrically $Q_n(r,i)$ is the orbit point closest to $\gamma'_{T_n}$ with the distance between them bounded above by the diameter of $X_{\text{thick}}$, i.e. uniformly bounded from above. This implies that along the sequence of times $2\log q_n = T_n$ the limit $n/T_n$ is $6 \log 2/\pi^2$. It should be pointed out that the number $N$ of horoballs that $\gamma'$ actually intersects till $T_n$ is less than or equal to $n$, and in fact $N/T_n$ will have a different limit as $T_n \to \infty$. 

The geodesic ray $\gamma$ from $x_0$ to $r$ and the vertical ray $\gamma'$ are asymptotic. Set 
\[
a = \max_{H \in \mathcal{H}} d_{\partial H}(\pi_H x_0, \pi_H (r,i))
\]
where $\pi_H$ is the closest point projection to $H$. Then we have the crude bound
\[
E(\gamma, H) - 2a e^{-\tau} - 2 < E(\gamma', H) < E(\gamma, H) + 2a e^{-\tau} + 2
\]
where $\tau$ is minimum of $d(x_0, H)$ and $d((r,i), H)$. Let $d$ be the distance between horocycles with $r$ at infinity that pass through $x_0$ and $(r,i)$ respectively. Then depending on the case we get the crude bound
\[
E(\gamma, T_n \pm d)-  2(a+1)(n+ 2\eta(T_n + d))   \leqslant \sum_{k \leqslant n } E(\gamma', H'_k) \leqslant E(\gamma, T_n \pm d) + 2(a+1)(n+ 2\eta(T_n + d)).
\] 
The estimate above implies that
\[
\lim_{n \to \infty} \frac{ \sum\limits_{k \leqslant n } E(\gamma', H'_k) - \max\limits_{k \leqslant n} E(\gamma', H'_k)}{T \log T} = \lim_{T_n \to \infty} \frac{E(\gamma, T_n \pm d) - \max\limits_{k\leqslant N} E(\gamma, H_k)}{T_n \log T_n} = \frac{6}{\pi^2}
\]
where the second equality follows from the fact that passing to $T_n$ instead $T_n \pm d$ in the numerator introduces an additive error that is at most $e^d$. 
Finally, note that $a_k - 1 \leqslant E(\gamma', H'_k) \leqslant a_k+1$ and so for Leb-almost every $r \in [0,1]$
\begin{eqnarray*}
\lim_{n \to \infty} \frac{\sum\limits_{k=1}^n a_k - \max\limits_{k \leqslant n} a_k}{n \log n} &=& \lim_{n \to \infty} \left( \frac{\sum\limits_{k \leqslant n} E(\gamma', H'_k) - \max\limits_{k \leqslant n}  E(\gamma', H'_k)}{T_n \log T_n}\right)\left( \frac{T_n \log T_n}{ n \log n }\right)\\
&=& \left(\frac{6}{\pi^2} \right)\left(\frac{\pi^2}{6 \log 2}\right) = \frac{1}{\log 2}
\end{eqnarray*} 
finishing the proof of Theorem \ref{Diam-Vaal}.

\end{proof}

\section{Partial sums along random \teichmuller geodesics in a stratum of quadratic differentials}

\subsection{Preliminaries from \teichmuller theory.}

Let $S$ be a hyperbolic surface of finite type, i.e. a surface of finite area which may have boundary components or punctures. We say such a surface $S$ is \emph{sporadic} if it is a sphere with at most
four punctures or boundary components, or a torus with at most one puncture or boundary component.  We shall primarily be interested in non-sporadic surfaces, as in the sporadic cases the Teichm\"uller
spaces are either trivial, or isometric to $\mathbb{H}^2$, which reduces us to the case of a non-uniform lattice in $SL(2, \R)$.

Let $S$ be a non-sporadic surface which has no boundary components, but may have punctures. The Teichm\"uller space $\T(S)$ is the space of marked conformal structures on $S$. Alternatively, by uniformization, it is the space of marked hyperbolic metrics on $S$. We shall consider
$\T(S)$ together with the Teichm\"uller metric
\[ d_{\mathcal{T}}(X, Y) = \tfrac{1}{2} \inf_f \log
K(f) \]
where the infimum is taken over all quasiconformal maps $f \colon X \to Y$, and $K(f)$ is the quasiconformal constant for the map $f$. The mapping class group $\textup{Mod}(S)$ acts by isometries on $\T(S)$. Let $\T(S)_\epsilon$ be the \emph{thin part} of Teichm\"uller space, i.e. all surfaces which contain a curve of hyperbolic length at most $\epsilon$. Let $\mathcal{M}(S)$ be the moduli space $\textup{Mod}(S) \backslash \T(S)$.  The thin part $\T(S)$ is $\textup{Mod}(S)$ invariant.  The thin part $\mathcal{M}(S)_\epsilon $ of moduli space is the quotient $\textup{Mod}(S) \backslash \T(S)_\epsilon$.

Let $\mathcal{Q}(X)$ be the unit area meromorphic quadratic differentials on $X$ with simple poles at all the punctures of $X$. If $(k_1, k_2, \dots, k_r)$ are the multiplicities of the zeros of a quadratic differential $q$ then $k_1+\dots + k_r  = 4g-4+2m$ where $m$ is the number of punctures of $X$. 
By contour integration, a quadratic differential $q$ defines a half-translation structure on $S$, i.e. it defines charts from $S$ to $\C$ with transition functions of the form $z \to \pm z+c$. The resulting flat metric has a cone singularity with cone angle $(k+2)\pi$ at a zero of $q$ order $k$ (or with $k=-1$ at a simple pole). A quadratic differential is unit area if the corresponding flat metric has unit area. The space $\mathcal{Q}$ of unit area quadratic differentials can be identified with the unit cotangent bundle to $\T(S)$ \cite{Hub-Mas}. We let $\pi : \mathcal{Q} \to \T(S)$ be the projection which sends a quadratic differential to its underlying Riemann surface. The space $\mathcal{Q}$ is stratified by the multiplicities of the zeros: we denote the strata with multiplicities $\alpha = (k_1, \dots, k_r)$ by $\mathcal{Q}(\alpha)$. For each stratum, the number of connected components is bounded. See \cite{Kon-Zor}, \cite{Boi-Lan}. 

In the flat metric defined by a quadratic differential $q$, a {\em saddle connection} is a geodesic segment that connects a pair of (same or distinct) singularities of $q$. The $\epsilon$-thin part, $\mathcal{Q}(\alpha)_\epsilon$ of $\mathcal{Q}(\alpha)$ is the subset of $q$ such that  $\ell^2_q(\beta) \leqslant \epsilon$ for some saddle connection $\beta$. The relationship between the thin parts of $\mathcal{Q}$ and $\T(S)$ is complicated: quadratic differentials on hyperbolic surfaces with short curves necessarily have short saddle connections but the converse need not be true.

For any $q \in \mathcal{Q}(\alpha)$ there is a canonical ramified double cover such that the lift of $q$ is square of a holomorphic 1-form $\omega$ and $(X, q)$ is a quotient of the double cover with respect to hyper-elliptic involution. Fixing a basis for the anti-invariant (with respect to hyper-elliptic involution) part of the homology of the double cover relative to the singularities, the holonomies (periods) given by integrating $\omega$ over the basis defines local co-ordinates on $\mathcal{Q}(\alpha)$. The natural volume form in these co-ordinates defines the Masur-Veech measure. Alternatively, it is known as the holonomy measure. We shall denote it by  $\mu_{\text{hol}}$. The measure $\mu_{\text{hol}}$ is $\textup{Mod}(S)$-invariant. So it descends to a measure on $\textup{Mod}(S) \backslash \mathcal{Q}(\alpha)$, the corresponding stratum of the moduli space of quadratic differentials. We continue to denote it by $\mathcal{Q}(\alpha)$. The $\mu_{\text{hol}}$-volume of $\mathcal{Q}(\alpha)$ is finite \cite{Mas1} \cite{Vee}. 

The affine action of $SL(2,\R)$ on the charts to $\C = \R^2$ preserves the glueing by half-translations. This defines an action of $SL(2,\R)$ on $\mathcal{Q}(\alpha)$. The orbits $SL(2, \R)(q)$ foliate $\mathcal{Q}(\alpha)$. The compact part $SO(2, \R)$ acts by rotations of $\R^2$. Hence, it preserves the conformal structure. The action of the diagonal subgroup defines the \teichmuller geodesic flow. It shrinks the leaves of the vertical foliation for $q$ and stretches the leaves of the horizontal foliation for $q$ by the same factor. It follows from the definition that $\mu_{\text{hol}}$ is $SL(2, \R)$-invariant. 

Since $SO(2,\R)$ preserves the conformal structure we get an isometrically embedded $\H^2 = SL(2, \R)/ SO(2, R)$ in $\T(S)$. This is called a Teichm\"uller disc and we will denote the \teichmuller disc determined by $q$ as $\D(q)$. The Teichm\"uller metric restricted to $\D(q)$ is isometric to the hyperbolic plane of constant curvature $-4$. 

The points $q'$ in $SL(2,\R)(q)$ where $\ell^2_{q'}(\beta) \leqslant \epsilon$ projects to a horoball in $\D(q)$. The point at infinity of the horoball is given by the direction in which $\beta$ is vertical. When two saddle connections $\beta_1$ and $\beta_2$ are parallel the proportion $[\ell_{q'}(\beta_1) :  \ell_{q'}(\beta_2) ]$ as a function of $q'$ is constant. Hence, the horoball is determined by the saddle connection with the shortest holonomy in a collection of parallel saddle connections. Typically, the intersection $\mathcal{Q}(\alpha)_\epsilon \cap \D(q)$ gives a complicated collection of horoballs in $\D(q)$: while every point in $\D(q)$ is contained in finitely many horoballs there need not be a uniform bound on this number. 

\subsection{$SL(2,\R)$ orbit closures and invariant measures.} Recently, Eskin and Mirzakhani \cite[Theorem 1.4]{Esk-Mir} showed that ergodic $SL(2,\R)$-invariant probability measures are of Lebesgue class and are supported on invariant complex submanifolds in $\mathcal{Q}(\alpha)$. These manifolds are affine in the sense that in holonomy co-ordinates they are given by linear equations. Going further, Eskin, Mirzakhani and Mohammadi in \cite[Theorem 2.1]{Esk-Mir-Moh} show that all $SL(2, \R)$ orbit closures are affine invariant submanifolds. See \cite[Section 1]{Esk-Mir} for more details. More recently, Filip \cite{Fil} shows that these submanifolds are in fact algebraic subvarieties. 

\subsection{Thin parts and regularity for invariant measures.} 
Let $\mu$ be an ergodic $SL(2,\R)$-invariant probability supported on an affine invariant submanifold $\mathcal{N} \subset \mathcal{Q}(\alpha)$. For $\epsilon > 0$, we define the $\epsilon$ thin part of $\mathcal{N}$ as follows:
\[
\mathcal{N}_{\epsilon} = \{ q \in \mathcal{N} \text{ such that } \ell_q^2(\beta) \leqslant \epsilon \text{ for some saddle connection } \beta\}.
\]
Saddle connections $\beta_1, \beta_2$ are $\mathcal{N}$-parallel if they are parallel for an open subset of quadratic differentials in $\mathcal{N}$. See Definition 4.6 in \cite{Wri}. 

\subsubsection*{Regularity:} For $\epsilon> 0, \kappa > 0$ small enough, let $\mathcal{N}_{\epsilon, \kappa}$ be the subset of $q \in \mathcal{N}$ such that there is a pair of saddle connections $\beta_1, \beta_2$ not $\mathcal{N}$-parallel such that $\ell^2_q(\beta_1) \leqslant \epsilon$ and $\ell^2_q(\beta_2) \leqslant \kappa$. The measure $\mu$  is said to be {\em regular} if there exists a constant $m_1 > 0$ such that 
\begin{equation}\label{regularity}
\mu(\mathcal{N}_{\epsilon, \kappa} ) \leqslant m_1 \epsilon \kappa
\end{equation}
Masur and Smillie \cite[Section 10, Claim (7)]{Mas-Smi} show that the holonomy measure $\mu_{\text{hol}}$ is regular. Avila, Matheus and Yoccoz \cite[Theorem 1.2]{Avi-Mat-Yoc} prove a weaker regularity for any $SL(2, \R)$-invariant measure. 

\subsection{$SL(2,\R)$-invariant loci, Siegel-Veech transform and volume asymptotic.}
For $q \in \mathcal{N}$, let $V(q) \subset \R^2 \setminus \{(0,0)\}$ be an assignment of a non-empty subset of holonomies of saddle connections on $q$. We require that the assignment varies linearly under $SL(2,\R)$ action, i.e. $V(g q) = g V(q)$ for all $g \in SL(2,\R)$. As observed in \cite{Esk-Mas}, such an assignment satisfies conditions $(B)$ and $C_\mu$ mentioned in their paper for any $SL(2,\R)$-invariant measure $\mu$. Such an assignment $V$ will be called a $SL(2,\R)$-invariant locus.

Let $f$ be a smooth function on $\R^2$ with compact support. The Siegel-Veech transform associated to a $SL(2,\R)$-invariant locus $V$ is defined as
\[
\widehat{f}(q) = \sum_{v \in V(q)} f(v)
\]
Veech showed that $f \in L^1(\mathcal{N}, \mu)$ and proved the Siegel-Veech formula
\begin{equation}\label{Siegel-Veech}
\int_{\mathcal{N}} \widehat{f} \, d\mu = c(V,\mu) \int_{\R^2} f \, dx \, dy
\end{equation}
where the constant $c(V,\mu)$ does not depend on $f$. The constant $c(V,\mu)$ is called the Siegel-Veech constant associated to $V$ and $\mu$. We assume that the assignment $V$ is such that $c(V,\mu) >0$. 

For $R \geqslant 1$, the $\epsilon/R$-thin part of $\mathcal{N}$ corresponding to $V$ is the set of $q$ with some saddle connection with holonomy in $V$ has $q$-length squared less than $\epsilon/R$.  We denote the set by $\mathcal{N}(V)_{\epsilon/R}$. 

Let $f_{\epsilon/R}$ be the characteristic function of the ball $B((0,0), \sqrt{\epsilon/R})$ centered at the origin and radius $\sqrt{\epsilon/R}$. While $f_{\epsilon/R}$ is not smooth the Siegel-Veech formula extends to such characteristic functions. The regularity condition \ref{regularity} and the Siegel-Veech formula \ref{Siegel-Veech} applied to $f_{\epsilon/R}$ can be used to prove the volume asymptotic
\begin{equation}\label{n-vol}
\lim_{R \to \infty} \frac{\mu(\mathcal{N}(V)_{\epsilon/R})}{\pi \epsilon/R} = c(V, \mu).
\end{equation}
See \cite[Section 7]{Esk-Mas-Zor} for the main ideas.

\subsection{Exponential mixing of \teichmuller flow:} 
It is known that the \teichmuller flow is exponentially mixing. For the Masur-Veech measure, the decay of correlations \ref{Ratner-decay} for $SO(2,\R)$-invariant $L^2$-functions is due to Avila-Gou\"ezel-Yoccoz \cite{Avi-Gou-Yoc} for holomorphic 1-forms and due to Avila-Resende \cite{Avi-Res} for quadratic differentials. For general $SL(2,\R)$-invariant measures this is due to Avila-Gou\"ezel \cite{Avi-Gou}. Since the functions we consider are pullbacks from \teichmuller discs they are $SO(2, \R)$-invariant. Hence, the decay of correlations applies to them.

\section{Proofs of Theorem \ref{saddles-V} }
\subsection*{The simplest case:} We first prove Theorem \ref{saddles-V} in the simplest case when $V(q)$ is the set of holonomies of {\em all} saddle connections on $q$. This allows us to convey the key ideas while getting into less subtleties. We denote the corresponding Siegel-Veech constant simply as $c(\mu)$. 

Let $q \in \mathcal{N}_\epsilon$. Consider short saddle connections in $q$. If some saddle connections are parallel we choose the one with the smallest holonomy among them. Suppose that lengths of these short saddle connections are given by $\ell^2_q(\beta_1) = \epsilon/R_1, \ell^2_q(\beta_2) = \epsilon/R_2, \dots, \ell^2_q(\beta_k) = \epsilon/R_k$ with $R_1 \geqslant R_2 \geqslant \dots \geqslant R_k \geqslant 1$.  We define 
\[
\Psi(q) = \frac{2}{\pi} R_1.
\]
Next we define
\[
\widehat{\Psi}(q) = \frac{2}{\pi}\left( R_1 + R_2 + \dots+ R_k \right).
\]
Obviously $\widehat{\Psi}(q) \geqslant \Psi(q)$ for all $q$. At first glance, the function $\Psi$ above is similar to the function $\Psi$ defined in the context of non-uniform lattices. However, here the sum over all excursions between successive entry and exit times $T_1 < T_2$ in $\mathcal{N}_\epsilon$ can satisfy
\[
\sum\limits_{H:  \gamma[T_1, T_2] \cap H \neq \emptyset} E(\gamma, H) \gg  \int_{T^1}^{T_2} \Psi(v(\gamma_t)) \, dt.
\]
This discrepancy is rectified by using the larger function $\widehat{\Psi}$. The key point is to estimate the difference in the $L^1$ and $L^2$ norms of the truncations of $\Psi$ and $\widehat{\Psi}$ in terms of the depth in $\mathcal{N}_\epsilon$ of the truncations. This will enable us to show that the above discrepancy does not happen too often.

Let $\chi_R$ denote the characteristic function of $\mathcal{N}_{\epsilon/R}$ and define the truncation
\[
\Psi_R = (\chi_1 - \chi_R) \Psi.
\]

\begin{lemma}\label{phi-norms}
\[
\lim_{R \to \infty} \frac{\Vert \Psi_R \Vert_{L^1}}{\log R}  =  2 \epsilon c(\mu) \hskip 8pt, \hskip 8pt \lim_{R \to \infty} \frac{\Vert \Psi_R  \Vert_{L^2}}{\sqrt{R}} = 2 \frac{\sqrt{\epsilon c(\mu)}}{\sqrt{\pi}}.
\]
\end{lemma}

\begin{proof}
It follows from \ref{n-vol} that for any $A > 1$ there is $R_0$ such that for all $R > R_0$ 
\[
\frac{1}{A} \frac{\pi \epsilon c(\mu)}{R}  < \mu (\mathcal{N}_{\epsilon/R}) < A \frac{\pi \epsilon c(\mu)}{R}.
\]
Fix $r > 0$ and for any positive integer $k$ consider $\mathcal{N}_{\epsilon/2^{(k-1)r}} \setminus \mathcal{N}_{\epsilon/2^{kr}}$. If $k$ is large enough so that $2^{(k-1)r} > R_0$ then the measure of the above set satisfies
\[
\frac{\pi \epsilon c(\mu)}{2^{kr}} \left( \frac{2^r-A^2}{A}\right)  < \mu \left(\mathcal{N}_{\epsilon/2^{(k-1)r}} \setminus \mathcal{N}_{\epsilon/2^{kr}} \right) < \frac{\pi \epsilon c(\mu)}{2^{kr}} \left( \frac{2^rA^2-1}{A}\right).
\]
Given $r$, we choose $A$ close to $1$ such that 
\begin{equation}\label{A-condition}
\frac{2^r -1}{2^r}  < \frac{2^r-A^2}{A} < \frac{2^rA^2 -1}{A} < 2^r(2^r -1). 
\end{equation}
Let $n$ be the largest integer such that $2^{nr} \leqslant R$. The $L^1$-norm of $\Psi_R$ can be estimated by
\begin{align*}
\frac{2}{\pi} \sum_{k=1}^n 2^{(k-1)r} \mu \left(\mathcal{N}_{\epsilon/2^{(k-1)r}} \setminus \mathcal{N}_{\epsilon/2^{kr}}  \right) &< \Vert \Psi_R \Vert_{L^1}\\
&< \frac{2}{\pi} \sum_{k=1}^{n+1} 2^{kr} \mu \left( \mathcal{N}_{\epsilon/2^{(k-1)r}} \setminus \mathcal{N}_{\epsilon/2^{kr}} \right).
\end{align*}
Let $n_0$ be the smallest integer such that $2^{n_0 r} \geqslant R_0$. We assume that $R \gg R_0$. The summation in the lower bound on the left satisfies
\begin{align*}
\sum_{k=0}^n 2^{(k-1)r} \mu \left(\mathcal{N}_{\epsilon/2^{(k-1)r}} \setminus \mathcal{N}_{\epsilon/2^{kr}}  \right) > &\sum_{k=1}^{n_0-1} 2^{(k-1)r} \mu \left(\mathcal{N}_{\epsilon/2^{(k-1)r}} \setminus \mathcal{N}_{\epsilon/2^{kr}}  \right) +\\  &\sum_{k=n_0}^n 2^{(k-1)r} \frac{\pi \epsilon c(\mu)}{2^{kr}} \left(\frac{2^r -1}{2^r} \right).
\end{align*}
The right hand side of the above inequality simplifies to
\[
\sum_{k=1}^{n_0-1} 2^{(k-1)r} \mu \left(\mathcal{N}_{\epsilon/2^{(k-1)r}} \setminus \mathcal{N}_{\epsilon/2^{kr}}  \right) + \frac{(n-n_0) \pi \epsilon c(\mu)}{2^r} \left( \frac{2^r-1}{2^r} \right).
\]
As $R$ becomes large, the second term dominates and since $(n-n_0)/ \log R \to  1/ r \log 2$ the above expression simplifies to
\[
\frac{2}{2^{2r} \log 2}\left(\frac{2^r - 1}{r} \right) \epsilon c(\mu) < \lim_{R \to \infty} \frac{\Vert \Psi_R \Vert_{L^1}}{ \log R} 
\]
which as $r \to 0$ implies
\[
2 \epsilon c(\mu) \leqslant \lim_{R \to \infty} \frac{\Vert \Psi_R \Vert_{L^1}}{\log R}.
\]
Similarly the summation for the upper bound on $\Vert \Psi_R \Vert_{L^1}$ gives
\[
\frac{2 }{2^r \log 2} \left( \frac{2^r - 1}{r} \right) \epsilon c(\mu) > \lim_{R \to \infty} \frac{\Vert \Psi_R \Vert_{L^1}}{\log R}
\]
which as $r \to 0$ implies 
\[
2 \epsilon c(\mu) \geqslant  \lim_{R \to \infty} \frac{\Vert \Psi_R \Vert_{L^1}}{\log R}
\]
proving the lemma for the $L^1$-norm. 

In a similar way, the square of the $L^2$-norm of $\Psi_R$ can be estimated by
\begin{align*}
\frac{4}{\pi^2} \sum_{k=1}^n 2^{2(k-1)r} \mu \left( \mathcal{N}_{\epsilon/2^{(k-1)r}} \setminus \mathcal{N}_{\epsilon/2^{kr}} \right) &< \Vert \Psi_R \Vert^2_{L^2} \\  &< \frac{4}{\pi^2} \sum_{k=1}^{n+1} 2^{2kr} \mu \left( \mathcal{N}_{\epsilon/2^{(k-1)r}} \setminus \mathcal{N}_{\epsilon/2^{kr}} \right). 
\end{align*}
The summation in the lower bound on the left satisfies
\begin{align*}
\sum_{k=1}^n 2^{2(k-1)r} \mu \left( \mathcal{N}_{\epsilon/2^{(k-1)r}} \setminus \mathcal{N}_{\epsilon/2^{kr}} \right) > &\sum_{k=1}^{n_0-1} 2^{2(k-1)r} \mu \left( \mathcal{N}_{\epsilon/2^{(k-1)r}} \setminus \mathcal{N}_{\epsilon/2^{kr}} \right) + \\ &\sum_{k=n_0}^n 2^{2(k-1)r} \frac{\pi \epsilon c(\mu)}{2^{kr}} \left( \frac{2^r- 1}{2^r} \right). 
\end{align*}
The right hand side of the inequality above is equal to
\[
\sum_{k=1}^{n_0-1} 2^{2(k-1)r} \mu \left( \mathcal{N}_{\epsilon/2^{(k-1)r}} \setminus \mathcal{N}_{\epsilon/2^{kr}} \right) + \left(\frac{2^{(n+1)r} - 2^{n_0 r}}{2^r - 1} \right)\frac{\pi \epsilon c(\mu)}{2^{2r}} \left(\frac{2^r - 1}{2^r} \right) 
\]
which is greater than
\[
\sum_{k=1}^{n_0-1} 2^{2(k-1)r} \mu \left( \mathcal{N}_{\epsilon/2^{(k-1)r}} \setminus \mathcal{N}_{\epsilon/2^{kr}} \right) + \frac{\pi \epsilon c(\mu) R}{2^{3r}} - \frac{2^{n_0 r} \pi \epsilon c(\mu)}{2^{3r}}.
\]
As $R$ becomes large the term containing $R$ dominates and letting $r \to 0$ we get  the lower bound
\[
\frac{2\sqrt{\epsilon c(\mu)}}{\sqrt{\pi}} \leqslant \lim_{R \to \infty} \frac{\Vert \Psi_R \Vert_{L^2}}{\sqrt{R}}. 
\]
Similarly the summation in the upper bound satisfies
\begin{align*}
\sum_{k=1}^{n+1} 2^{kr} \mu \left( \mathcal{N}_{\epsilon/2^{(k-1)r}} \setminus \mathcal{N}_{\epsilon/2^{kr}} \right) <  &\sum_{k=1}^{n_0-1}  2^{2kr} \mu \left( \mathcal{N}_{\epsilon/2^{(k-1)r}} \setminus \mathcal{N}_{\epsilon/2^{kr}} \right) + \\ &\sum_{k= n_0}^{n+1}  2^{2kr} \frac{\pi \epsilon c(\mu)}{2^{(k-1)r}} \left( 2^r (2^r -1) \right) \\
< & \sum_{k=1}^{n_0-1}  2^{2kr} \mu \left( \mathcal{N}_{\epsilon/2^{(k-1)r}} \setminus \mathcal{N}_{\epsilon/2^{kr}} \right)  + 2^{4r}R \pi \epsilon c(\mu). 
\end{align*}
As $R$ becomes large the term containing $R$ dominates and letting $r \to 0$ we get the upper bound
\[
\frac{2\sqrt{\epsilon c(\mu)}}{\sqrt{\pi}} \geqslant \lim_{R \to \infty} \frac{\Vert \Psi_R \Vert_{L^2}}{\sqrt{R}}.
\]
finishing the proof for the $L^2$-norm.
\end{proof}
The next lemma will need the regularity \ref{regularity} for $\mu$ and the following theorem of Eskin and Masur: For any stratum $\mathcal{Q}(\alpha)$ and any $0 < \delta < 1$ there exists constants $m_2(\alpha, \delta) > 0$ such that the number of saddle connections shorter than $\kappa$ is bounded above by
\begin{equation}\label{number-saddles}
s(q) \leqslant m_2 \left(\frac{\kappa}{\ell_q(\beta)} \right)^{1+ \delta}
\end{equation}
where $\beta$ is the shortest saddle connection for $q$. It should be noted that while Eskin and Masur state the theorem for strata of holomorphic 1-forms it is true for strata of quadratic differentials by passing to the canonical double cover. 

Fix the constant $\delta$ in the Eskin-Masur theorem and choose $a> 1$ such that $a < 2/(1+ \delta)$. Let $\mathcal{N}' \subset \mathcal{N}_{\epsilon/R} \setminus \mathcal{N}_{\epsilon/R^a}$ be the subset of quadratic differentials such that apart from the shortest saddle connection, all other short saddle connections satisfy $\ell^2_q(\beta_i) \geqslant \epsilon/R$. We define the function $\Psi' : \mathcal{N}' \to \R_{\geqslant 0}$ by
\[
\Psi'(q) = \frac{2}{\pi} ( R_2 + R_3 + \dots + R_j).
\]
We define a slightly more complicated truncation
\[
\widehat{\Psi}_R = \widehat{\Psi}(\chi_1 - \chi_R) + \Psi'.
\]
As we shall see in the proof of Theorem \ref{saddles-V} in this particular case, the extra term $\Psi'$ allows us to keep track of excursions that are concurrent with the largest excursion if it exceeds $T(\log T)^c$.

\begin{lemma}\label{phihat-norms}
There exists a constant $B> 0$ such that for $R$ large enough
\[
\Vert \widehat{\Psi}_R \Vert_{L^1} \leqslant \Vert \Psi_R \Vert_{L^1} + B.
\]
The $L^2$-norms satisfy
\[
\Vert \widehat{\Psi}_R \Vert_{L^2} \leqslant \Vert \Psi_R \Vert_{L^2} + o(\sqrt{R}).
\]
\end{lemma}

\begin{proof}

Consider the subset $\mathcal{N}_{k,j}$ of $\mathcal{N}_{\epsilon/ 2^{k-1} }\setminus \mathcal{N}_{\epsilon/2^k}$ of those $q$ such that the length of the second shortest saddle connection $\beta_2$ (not $\mathcal{N}$-parallel to $\beta_1$) satisfies $\epsilon/2^{j-1} \geqslant \ell_q^2(\beta_2) > \epsilon/2^j$ where $j \leqslant k$. By regularity \ref{regularity}, 
\[
\mu (\mathcal{N}_{k,j}) \leqslant \frac{m_1 \epsilon^2}{2^{k-1}2^{j-1}}.
\]
For a quadratic differential $q$ in $\mathcal{N}_{k,j}$, using the bound \ref{number-saddles} on the number of short saddles we get the pointwise bound 
\[
\widehat{\Psi}_R(q) - \Psi_R(q) < \sum_{i=1}^j 2^{i}m_2 \left(\frac{2^k}{2^{i-1}}\right)^{(1+ \delta)/2} < 2m_2 2^{k(1+ \delta)/2} \sum_{i=1}^j 2^{(i-1)(1- \delta)/2} < m_3 2^{k(1+ \delta)/2} 2^{j(1-\delta)/2}
\]
for some constant $m_3 > 0$. Thus 
\[
\int_{\mathcal{N}_{k,j}} (\widehat{\Psi}_R- \Psi_R) \, d\mu < \left(m_3 2^{k(1+ \delta)/2} 2^{j(1-\delta)/2}\right) \left( \frac{m_1 \epsilon^2}{2^{k-1} 2^{j-1}}\right) = \frac{4m_1 m_3 \epsilon^2}{2^{k(1- \delta)/2} 2^{j(1+ \delta)/2}}.
\]
Summing over $j=1$ to $k$ (when $2^{k-1} > R$ it suffices to sum till the smallest number $k'$ such that $2^{k'} > R$ in which case the sum would be even smaller) we get 
\[
\int_{\mathcal{N}_{\epsilon/2^{k-1}} \setminus \mathcal{N}_{\epsilon/2^k}} (\widehat{\Psi}_R- \Psi_R) \, d\mu < \frac{m_4}{2^{k(1-\delta)/2}}
\]
for some constant $m_4 > 0$. Let $n_a$ be the smallest integer such that $2^{n_a} \geqslant R^a$. The bound for the integral established above implies that 
\[
\Vert \widehat{\Psi}_R \Vert_{L^1} - \Vert \Psi_R \Vert_{L^1} < \sum_{k=1}^{n_a} \frac{m_4}{2^{k(1-\delta)/2}}.
\]
The sum on the right hand side is bounded from above independent of $n_a$ which proves the lemma for $L^1$ norms.

The same pointwise bound above implies
\begin{align*}
\int_{\mathcal{N}_{k,j}} \widehat{\Psi}^2_R \, d\mu &< \int_{\mathcal{N}_{k,j}} \Psi^2_R \, d\mu +  2m_3 2^{k(1+ \delta)/2} 2^{j(1- \delta)/2}\int_{\mathcal{N}_{k,j}} \Psi_R \, d\mu + m^2_3 2^{k(1+ \delta)}2^{j(1- \delta)}\mu (\mathcal{N}_{k,j})\\
&< \int_{\mathcal{N}_{k,j}} \Psi^2_R \, d\mu + 2m_3 2^{k(1+ \delta)/2} 2^{j(1- \delta)/2}\left( \frac{2^k m_1 \epsilon^2}{2^{k-1}2^{j-1}}\right) + m_3^2 2^{k(1+ \delta)} 2^{j(1- \delta)}\frac{m_1\epsilon^2}{2^{k-1}2^{j-1}}\\
&=  \int_{\mathcal{N}_{k,j}} \Psi^2_R \, d\mu + \frac{1}{2^{j(1+ \delta)/2}} 8m_1m_3 \epsilon^2 2^{k(1+ \delta)/2} + \frac{1}{2^{j\delta}} 4m_1 m^2_3 \epsilon^2 2^{k\delta}
\end{align*}
Summing over $j = 1$ to $k$ we get
\[
\int_{\mathcal{N}_{\epsilon/2^{k-1}} \setminus \mathcal{N}_{\epsilon/2^k}} \widehat{\Psi}^2_R \, d\mu < \int_{\mathcal{N}_{\epsilon/2^{k-1}} \setminus \mathcal{N}_{\epsilon/2^k}} \Psi^2_R \, d\mu  + m_5 2^{k(1+ \delta)/2} + m_6 2^{k\delta}.
\]
for some constants $m_5, m_6 > 0$. Summing over $k=1$ to $n_a$ we get
\[
\Vert \widehat{\Psi}_R \Vert^2_{L^2} < \Vert \Psi_R  \Vert^2_{L^2} + \frac{m_7}{2} 2^{n_a(1+ \delta)/2} + \frac{m_8}{2} 2^{n_a  \delta} < \Vert \Psi_R \Vert^2_{L^2} + m_7 R^{a(1+ \delta)/2} + m_8 R^{a \delta} 
\]
for some constants $m_7, m_8 > 0$. Recall that we had chosen $a> 1$ to satisfy $a(1+ \delta) < 2$ which implies $a \delta < a(1+ \delta)/2 < 1$. Thus, the corresponding terms on the right hand side are $o(R)$ from which the lemma follows for $L^2$-norms.
\end{proof}

We will justify the choice of the cutoff $R^a$ for truncation in the following lemma which is a continuous time version of the analog of Borel-Bernstein theorem \cite[Theorem 30]{Khi}, \cite{Ath-Par-Tse}. 

\begin{lemma}\label{lemma:a}
For any $a>1$ and $\mu$-almost every $q \in \mathcal{N}$ there is $T_0$ depending on $q$ such that for all $T > T_0$, all excursions $E(\gamma, H)$ till time $T$ satisfy 
\[
E(\gamma, H) < T^a.
\]
\end{lemma}
\begin{proof}
Choose $A$ close to 1 such that estimate \ref{A-condition} is satisfied with $r=1$. Let $\phi_R$ be the characteristic function of $\mathcal{N}_{\epsilon/(R/2)} \setminus \mathcal{N}_{\epsilon/R}$.The choice of $A$ implies that if $k$ is large enough then
\[
\Vert \phi_{R} \Vert_{L^1} < \frac{2\pi \epsilon c(\mu)}{R} \hskip 8pt \text{and hence} \hskip 8pt \Vert \phi_{R} \Vert^2_{L^2}  < \frac{2\pi \epsilon c(\mu)}{R}.
\]
For $T$ satisfying $2^k \leqslant T < 2^{k+1}$ set $n(T) = a(k-1)$. Fixing $c> 1/2, m>1$, Theorem \ref{Esum} implies that $\mu$-almost every $q$ satisfies 
the bound
\begin{align*}
\int_0^T \phi_{2^n}(v(\gamma_t)) \, dt  &\leqslant mT \Vert \phi_{2^n} \Vert_{L^1} + T^{1/2}(\log T)^c \left( \Vert \phi_{2^n} \Vert^2_{L^2} - \Vert \phi_{2^n} \Vert^2_{L^1} \right)^{1/2} \\
&\leqslant m T \frac{2\pi \epsilon c(\mu)}{2^{a(k-1)}} + T^{1/2}(\log T)^c \left( \frac{2 \pi \epsilon c(\mu)}{2^{a(k-1)}} \right)^{1/2}\\
&\leqslant \frac{B_1}{2^{(a-1)k}} + \frac{B_2 k^c}{2^{(a-1)k/2}}.
\end{align*}
for some constants $B_1, B_2 > 0$. Notice that if $T$ and consequently $k$ is large enough then the right hand side is less than $\log 2$. But if an excursion satisfies $E(\gamma, H) > T^a$ then $\gamma$ must spend time at least $\log 2$ in $\mathcal{N}_{\epsilon/2^{n-1}} \setminus \mathcal{N}_{\epsilon/2^n}$, which proves the lemma. 
\end{proof}

\begin{proof}[Proof of Theorem \ref{saddles-V} when $V$ is all saddle connection holonomies]
Fix $c$ satisfying $1/2< c< 1$. Observe that Proposition \ref{two-large} which asserts that along typical geodesics there is at most a single excursion larger than $T(\log T)^c$ till time $T$ for $T$ large enough holds for $\mu$-typical \teichmuller geodesics. In fact, as noted in Remark \ref{tail} we do not need the precise asymptotic for $\mu (\mathcal{N}_{\epsilon/R})$ as $R \to \infty$. Moreover, by the previous lemma the largest excursion is smaller than $T^a$. 

If $2^k \leqslant T < 2^{k+1}$, then let $n =n(T) = \lfloor k + c \log_2 k \rfloor$. Replicating the exact argument in the proof of Theorem \ref{Esum}, we use Lemma 5.1 to conclude that for $\mu$-almost every $q \in \mathcal{N}$
\[
\lim_{T \to \infty} \frac{1}{T \log T}  \int_0^T \Psi_{2^n}(v(\gamma_t)) \, dt = 2\epsilon c(\mu)
\]
where $\gamma$ is the \teichmuller geodesic ray with $v(\gamma_0) = q$. Lemma \ref{phihat-norms} implies that the above limit holds when $\Psi_{2^n}$ is replaced by $\widehat{\Psi}_{2^n}$. Finally, up to an additive error whose dependence on $T$ will be described below
\[
E(\gamma, T) - \max_{k \leqslant N} E(\gamma, H_k)  \asymp \int_0^T \widehat{\Psi}_{2^n} (v(\gamma_t)) \, dt.
\]
By \ref{partial-exc} (and as pointed out in the proof of Theorem \ref{Esum}), the additive error from the partial excursion (if it exists) is bounded by $2 T (\log T)^c$. The additive error from complete excursions is at most linear in the number $N$ of horoballs that $\gamma$ intersects till time $T$ which we claim grows linearly in $T$. Let $\mathcal{S}_T$ be the saddle connections for $q$ whose length squared gets shorter than $\epsilon$ in time less than $T$ along the \teichmuller geodesic ray determined by $q$. Then the necessary conditions on the $(x,y)$-coordinates of the holonomy vectors of the saddle connections in $\mathcal{S}_T$ are $\vert x  y \vert \leqslant \epsilon^2/2, y \leqslant e^{T}$ and $x <  \epsilon$. By \cite[Theorem 1.6]{Ath-Par-Tse} the number of such vectors for a $\mu$-almost every $q$ is linear in $T$ thus proving the claim.
\end{proof}

\subsection{The general case.}
Let $V$ be a $SL(2,\R)$-invariant locus. Suppose for $q$ there are $j$ short saddle connections no two of which are $\mathcal{N}$-parallel with holonomy in $V(q)$. Index the saddle connections  $\beta_1, \beta_2, \dots, \beta_j$ in the order of increasing $q$-lengths and let $\ell^2_q(\beta_1) = \epsilon/R_1, \ell^2_q(\beta_2) = \epsilon/R_2, \dots, \ell^2_q(\beta_j) = \epsilon/R_j$. We define the functions
\[
\Psi^{V} (q) = \frac{2}{\pi} R_1
\]
and 
\[
\widehat{\Psi}^{V} (q) = \frac{2}{\pi}(R_1+ R_2 + \dots + R_j).
\]
We can define truncations of these functions in an analogous way using the characteristic function of $\mathcal{N}(V)_{\epsilon/R}$. However, the shortest saddle connection $\beta$ for $q$ may have holonomy not in $V(q)$ and be shorter than $\beta_1$ as above. Also, there is no lower bound on how short $\beta$ can be. Hence, some care is required in defining the truncations. The crucial point is that in light of Lemma \ref{lemma:a}, we can impose a lower bound on the length of  $\beta$ in defining the truncations.  

Let $a >1$ be such that $a < 2/(1+ \delta)$. In particular, $a\delta < 1$. Let $\mathcal{N}(R, a) \subset \mathcal{N}(V)_{\epsilon} \setminus \mathcal{N}(V)_{\epsilon/R}$ be the subset of those $q$ such that the shortest saddle connection $\beta$ satisfies $\ell^2_q(\beta) \geqslant \epsilon/R^a$. Let $\chi_{R,a}$ denote its characteristic function of $\mathcal{N}(R,a)$. 

Let $\mathcal{N'}(V) \subset \mathcal{N}(V)_{\epsilon/R} \setminus \mathcal{N}(V)_{\epsilon/R^a}$ be the subset of quadratic differentials such that $\ell^2_q(\beta) \geqslant \epsilon/R^a$ and apart from $\beta_1$ all other short saddle connections  with holonomy in $V(q)$ satisfy $\ell^2_q(\beta_i) \geqslant \epsilon/R$. Let $(\Psi^V)': \mathcal{N}'(V) \to \R$ be defined as 
\[
(\Psi^{V})'(q) = \frac{2}{\pi} (R_2 + \dots + R_j).
\]
We define $\Psi^{V}_R = \Psi^{V} \chi_{R,a} $ and $\widehat{\Psi}^{V}_R = \widehat{\Psi}^{V}\chi_{R,a}+ (\Psi^{V})'$. Again, the extra term $(\Psi^V)'$ analogous to $\Psi'$ earlier, is to allow us to track excursions for saddle connections in $V$ that are concurrent with the largest excursion with holonomy in $V$ if it exceeds $T(\log T)^c$.

For $2^k \leqslant R$ we have the estimate
\[
\mu \left( \mathcal{N}(V)_{\epsilon/2^{k-1}} \setminus \mathcal{N}(V)_{\epsilon/2^k} \right) - \mu \left( \mathcal{N}(R, a) \cap \mathcal{N}(V)_{\epsilon/2^{k-1}} \setminus \mathcal{N}(V)_{\epsilon/2^k} \right) \leqslant \frac{m_1 \epsilon^2}{2^{k-1} R^a}.
\]
This means here each term in the summations for lower and upper bound for $L^1$-norm in Lemma \ref{phi-norms} changes by at most $2^k m_1 \epsilon^2/2^{k-1} R^a = 2m_1 \epsilon^2/ R^a$. Hence the summations change by at most $2n m_1 \epsilon^2/R^a <  m_6 \log R/ R$ for some constant $m_6 > 0$. This implies 
\[
\lim_{R \to \infty} \frac{\Vert \Psi^{V}_R \Vert_{L^1}}{\log R} = 2 \epsilon c(V,\mu).
\]
Similarly each term in the summations for lower and upper bound for $L^2$ norms changes by at most $2^{2k}m_1 \epsilon^2/2^{k-1} R^a = m_1 \epsilon^2 2^{k+1}/R^a$ and hence the summations change by at most $4 m_1 \epsilon^2 2^n/R^a  <  m_7/R^{a-1}$ for some constant $m_7 > 0$. This implies 
\[
\lim_{R \to \infty} \frac{\Vert \Psi^{V}_R \Vert_{L^2}}{\sqrt{R}} = 2 \frac{\sqrt{\epsilon c(V, \mu)}}{\sqrt{\pi}}.
\]

\begin{lemma}\label{c-phihat}
There exists a constant $B_{V}$ such that for $R$ large enough
\[
\Vert \widehat{\Psi}^{V}_R \Vert_{L^1} \leqslant \Vert \Psi^{V}_R \Vert_{L^1} + B_{V}.
\]
The $L^2$-norms satisfy
\[
\Vert \widehat{\Psi}^{V}_R \Vert_{L^2} \leqslant \Vert \Psi^{V}_R \Vert_{L^2} + o(\sqrt{R}).
\]
\end{lemma}
\begin{proof}
Consider $\mathcal{N}(R, a) \cap \mathcal{N}(V)_{\epsilon/2^{k-1}} \setminus \mathcal{N}(V)_{\epsilon/2^k}$ and let $\mathcal{N}^{V}_{k,j}$ be its subset consisting of those $q$ for which (among the collection of non-parallel saddle connections with holonomies in $V(q)$) the second shortest saddle connection $\beta_2$ satisfies $ \epsilon/2^{j-1} \geqslant \ell^2_q(\beta_2) > \epsilon/2^j$ where $j \leqslant k$. 

We further partition $\mathcal{N}^{V}_{k,j}$ into two sets $\mathcal{N}^{V}_{k,j}(1) \cup \mathcal{N}^{V}_{k,j}(2)$ depending on whether the shortest saddle connection for $q$ has holonomy in $V(q)$ or not, i.e. $\mathcal{N}^{V}_{k,j}(1)$ is the subset of $q$ for which $\beta_1$ is the shortest saddle connection and $\mathcal{N}^{V}_{k,j}(2)$ is when its not. On $\mathcal{N}^{V}_{k,j}(1)$ the integral
\[
\int_{\mathcal{N}^{V}_{k,j}(1)}  \left( \widehat{\Psi}^{V}_R - \Psi^{V}_R \right) \, d\mu
\]
is bounded from above identical to Lemma 5.5.

Let $n_a$ be the smallest integer such that $2^{n_a} \geqslant R^a$. For $q \in \mathcal{N}^{V}_{k,j}(2)$ suppose that the shortest saddle connection $\beta$ satisfies $\epsilon/2^{p-1} > \ell^2_q(\beta) \geqslant \epsilon/2^p$ where $2^k \leqslant 2^p \leqslant 2^{n_a} $. The measure of the subset of such $q$ is bounded above by
\[
\frac{m_1 \epsilon^2}{2^{p-1} 2^{k-1}}.
\]
The number of short saddle connections whose $q$-length squared is at least $\epsilon/2^{i-1}$ is bounded above by
\[
m_2 \left( \frac{2^p}{2^{i-1}}\right)^{(1+ \delta)/2}.
\]
This gives the pointwise bound
\[
\widehat{\Psi}^{V}_R (q) - \Psi^{V}_R(q) < \sum_{i=1}^j 2^i m_2 \left(\frac{2^p}{2^{i-1}} \right)^{(1+ \delta)/2} = 2m_2 2^{p(1+\delta)/2} \sum_{i=1}^j 2^{(i-1)(1- \delta)/2} <  m_9 2^{p(1+ \delta)/2} 2^{j(1-\delta)/2}
\]
for some constant $m_9 > 0$. This gives the bound
\[
\int_{\mathcal{N}^{V}_{k,j}(2)} \left( \widehat{\Psi}^{V}_R - \Psi^{V}_R \right) \, d\mu < \sum_{p=k}^{n_a} m_9 2^{p(1+ \delta)/2} 2^{j(1-\delta)/2}\left( \frac{m_1 \epsilon^2}{2^{p-1}2^{k-1}} \right) < \frac{m_{10}}{2^k}
\]
for some constant $m_{10} > 0$. Thus adding up the upper bounds for the integrals on $\mathcal{N}^{V}_{k,j}(1)$ and $\mathcal{N}^{V}_{k,j}(2)$ we get 
\[
\int_{\mathcal{N}^{V}_{k,j}} \left( \widehat{\Psi}^{V}_R - \Psi^{V}_R \right) \, d\mu  < \frac{4m_1 m_3\epsilon^2}{2^{k(1-\delta)/2}2^{j(1+ \delta)/2}} + \frac{m_{10}}{2^k}
\]
Summing over $j =1$ to $k$ we get
\[
\int_{\mathcal{N}(R, a) \cap \mathcal{N}(V)_{\epsilon/2^{k-1}} \setminus \mathcal{N}(V)_{\epsilon/2^k}} \left( \widehat{\Psi}^{V}_R - \Psi^{V}_R \right) \, d\mu  < \frac{m_{11}}{2^{k(1-\delta)/2}} + \frac{m_{10}k}{2^k}
\]
for some constant $m_{11} > 0$. Summing over $k=1$ to $n_a$ observe that the sum of the right hand side is bounded independent of $n$ which proves the lemma for $L^1$-norms. 

The pointwise bound also implies
\begin{align*}
\int_{\mathcal{N}^{V}_{k,j}(2)} \left( \widehat{\Psi}^{V}_R \right)^2 - \left(\Psi^{V}_R \right)^2 \, d\mu &< \sum_{p=k}^{n_a} 2m_9 2^{p(1+ \delta)/2} 2^{j(1-\delta)/2} \left(\frac{2^k m_1 \epsilon^2}{2^{p-1} 2^{k-1}} \right)\\ &\hskip 12pt + \sum_{p=k}^{n_a} m^2_9 2^{p(1+ \delta)} 2^{j(1-\delta)}\left( \frac{m_1 \epsilon^2}{2^{p-1}2^{k-1}} \right)\\
&< \frac{m_{12}}{2^k} + \frac{m_{13} 2^{n_a \delta}}{2^{k\delta}}
\end{align*} 
for some constants $m_{12}, m_{13} > 0$. The corresponding upper bound for $\mathcal{N}^{V}_{k,j}(1)$ is identical to Lemma \ref{phihat-norms} and is of the form
\[
\int_{\mathcal{N}^{V}_{k,j}(1)} \left( \widehat{\Psi}^{V}_R \right)^2 - \left(\Psi^{V}_R \right)^2 \, d\mu  < \frac{m_{14} 2^{k(1+ \delta)/2}}{2^{j(1+\delta)/2}} + \frac{m_{15} 2^{k\delta}}{2^{j\delta}}
\]
for some constants $m_{14}, m_{15} > 0$. Adding up the bounds for $\mathcal{N}^{V}_{k,j}(1)$ and $\mathcal{N}^{V}_{k,j}(2)$ and summing over $j=1$ to $k$ we get
\[
\int_{\mathcal{N}(R, a) \cap \mathcal{N}(V)_{\epsilon/2^{k-1}} \setminus \mathcal{N}(V)_{\epsilon/2^k}}  \left( \widehat{\Psi}^{V}_R \right)^2 - \left(\Psi^{V}_R \right)^2 \, d\mu  < \frac{m_{12} k}{2^k} + \frac{m_{13} k 2^{n_a \delta}}{2^{k\delta}} + m_{14} 2^{k(1+ \delta)/2} + m_{15} 2^{k \delta}.
\]
and when $2^{k-1} > R$,
\[
\int_{\mathcal{N}'(V) \cap \mathcal{N}(V)_{\epsilon/2^{k-1}} \setminus \mathcal{N}(V)_{\epsilon/2^k}} \left( \widehat{\Psi}^{V}_R \right)^2 - \left(\Psi^{V}_R \right)^2 \, d\mu   < \frac{m_{12} k}{2^k} + \frac{m_{13} k 2^{n_a \delta}}{2^{k\delta}} + m_{14} 2^{k(1+ \delta)/2} + m_{15} 2^{k \delta}
\]
Summing over $k=1$ to $n_a$ we get that 
\[
\Vert \widehat{\Psi}^{V}_R \Vert^2_{L^2} - \Vert \Psi^{V}_R \Vert^2_{L^2} < \frac{m_{16}}{2} 2^{n_a \delta} + \frac{m_{17}}{2} 2^{n_a(1+ \delta)/2} + m_{18} < m_{16} R^{a \delta} + m_{17} R^{a(1+ \delta)/2} + m_{18}
\]
for some constants $m_{15}, m_{16}, m_{17}, m_{18} > 0$. The condition on $a$ implies that the right hand side is $o(R)$. Thus the lemma follows for $L^2$-norms.
\end{proof} 

\begin{proof}[Proof of Theorem \ref{saddles-V}]
Fix $c$ satisfying $1/2< c< 1$. For the same reason as in the proof of Theorem \ref{saddles-V} when $V$ is all saddle connection holonomies, Proposition \ref{two-large} holds for $\mathcal{N}(V)$ asserting that for $\mu$-almost every $q$, the \teichmuller geodesic ray corresponding to $q$ has at most a single excursion till time $T$, larger than $T(\log T)^c$ for all $T$ large enough depending on $q$. Moreover, for any $a> 1$, by Lemma \ref{lemma:a} the largest excursion cannot exceed $T^a$. 

The later fact implies that up to additive error our truncation $\widehat{\Psi}^{V}_{2^n}$ satisfies 
\[
\int_0^T \widehat{\Psi}^{V}_{2^n}(v(\gamma_t)) \, dt \asymp E(\gamma, T) - \max_{k \leqslant N_{V}} E(\gamma, H_k)
\]
where for the same reason as earlier the additive error is at most linear in $T$. Theorem \ref{ergodic-rate} and Lemma \ref{c-phihat} conclude the proof of Theorem \ref{saddles-V} in the general case, the precise argument a replica of earlier proofs.

\end{proof}


\end{document}